\documentclass[a4,10pt,DIV=9]{scrartcl}

\usepackage{amssymb,amsmath,mathtools,amsthm,mathrsfs}
\usepackage{color}

\usepackage[hidelinks,pdfusetitle]{hyperref}

\usepackage[capitalise]{cleveref}
\usepackage[backend=bibtex,maxbibnames=10]{biblatex}
\usepackage{enumitem}
\setlist[itemize,1]{label=\textbullet}
\setlist{nosep}

\newtheorem{thm}{Theorem}
\newtheorem{prop}[thm]{Proposition}
\newtheorem{cor}[thm]{Corollary}
\newtheorem{lemma}[thm]{Lemma}
\newtheorem{remark}[thm]{Remark}
\newtheorem{defi}[thm]{Definition}
\newtheorem*{MQ}{Main question}

\newcommand{\N}{\mathbb{N}} 
\newcommand{\R}{\mathbb{R}} 
\newcommand{\pR}{\mathbb{R}_{>0}} 
\newcommand{\nnR}{\mathbb{R}_{\ge 0}} 
\newcommand{\T}{\mathbb{T}} 
\newcommand{\domain}{\Omega} 
\newcommand{\fsL}{\textnormal{L}} 
\renewcommand{\H}{\textnormal{H}} 
\newcommand{\fsC}{\mathscr{C}} 
\newcommand{\weakto}{\rightharpoonup}
\newcommand{\ep}{\varepsilon}
\newcommand{\csubset}{\subset\joinrel\subset} 

\newcommand{\dd}{\mathrm{d}}
\newcommand{\lap}{\Delta}
\DeclareMathOperator{\divergence}{div}
\DeclareMathOperator{\supp}{supp}

\newcommand{\init}{\mathrm{init}}

\newcommand{\conv}{\star}

\newcommand{\sExt}{\mathcal{E}}

\author{Helge Dietert\footnote{Email: \href{mailto:helge.dietert@imj-prg.fr}{helge.dietert@imj-prg.fr}\newline
    Universit\'e de Paris and Sorbonne Universit\'e, CNRS,
    Institut de Math\'ematiques de Jussieu-Paris Rive Gauche (IMJ-PRG),
    F-75013, Paris, France\newline
    Currently on leave and working at\newline
    Institut f\"ur Mathematik, Universit\"at Leipzig, D-04103 Leipzig, Germany
  } \and Ayman Moussa\footnote{
    Email: \href{mailto:ayman.moussa@sorbonne-universite.fr}{ayman.moussa@sorbonne-universite.fr}\newline
   Sorbonne Université and Université de Paris, CNRS, Laboratoire Jacques-Louis Lions (LJLL), F-75005 Paris, France}}
\title{Persisting entropy structure for nonlocal cross-diffusion systems}

\begin{document}
\maketitle
\begin{abstract}
  For cross-diffusion systems possessing an entropy (i.e.\ a Lyapunov
  functional) we study nonlocal versions and exhibit sufficient
  conditions to ensure that the nonlocal version inherits the entropy
  structure.  These nonlocal systems can be understood as population
  models \emph{per se} or as approximation of the classical ones. With
  the preserved entropy, we can rigorously link the approximating
  nonlocal version to the classical local system. From a modelling
  perspective this gives a way to prove a derivation of the model and
  from a PDE perspective this provides a regularisation scheme to
  prove the existence of solutions. A guiding example is the SKT model
  \cite{shigesada-kawasaki-teremoto-1979-interacting-species}.  In
  this context we answer positively the question raised by
  \textcite{fontbona-meleard-2014-non-local-cross-diffusion} for the
  derivation and thus complete the derivation.
\end{abstract}

\section{Introduction}
\label{sec:introduction}

\subsection{Cross-diffusion systems with entropy structure}
Our starting points are cross-diffusion systems of \(n\) species with
densities \(u=(u_i)_{1\leq i\leq n}\) solving a system
\begin{equation}
  \label{eq:general-local-cross-diffusion}
  \partial_t u_i
  - \divergence \left(
    \sum_{j=1}^{n} a_{ij}(u) \nabla u_j
  \right) = 0,\qquad
  \text{for } i=1,\dots,n,
\end{equation}
on a domain \(\domain\) supplemented with boundary conditions and
initial data \(u^\init\). Here \(a_{ij}\) are given scalar functions
(\(\nnR^n \rightarrow \nnR\)) and the unknowns are the
model densities \(u_i\)'s, which are therefore expected to be
non-negative. The matrix $A(u):=(a_{ij}(u))$ is called the
\emph{diffusion} matrix and is always assumed to be positive
definite. As this work focuses on the entropy structure for the
diffusion, we do not consider here any reaction terms.

\medskip

Without any assumptions on the $a_{ij}$'s, the only estimate that we
have on system \eqref{eq:general-local-cross-diffusion} is the
conservation of the overall mass, \emph{i.e.}
\begin{equation*}
\frac{\dd}{\dd t}  \int_{\domain} u_i = 0,
\end{equation*}
for \(i=1,\dots,n\). Due to the severe non-linearity of the system,
this sole control is not sufficient to obtain the existence of global
solutions. Searching for a Lyapunov functional of the form
\begin{equation}
  \label{eq:general-additive-entropy}
  H(u): = \int_{\domain} \sum_{i=1}^n h_i(u_i),
\end{equation}
where $h_i \in \fsC^0(\nnR) \cap \fsC^2(\pR)$, we find formally
without boundary terms that
\begin{equation*}
  \frac{\dd}{\dd t} H(u)
  = - \int_{\domain}
  \begin{pmatrix}
    \nabla u_1\\
    \vdots \\
    \nabla u_n\\
  \end{pmatrix}
  \cdot
  M(u)
  \begin{pmatrix}
    \nabla u_1\\
    \vdots \\
    \nabla u_n\\
  \end{pmatrix}
,
\end{equation*}
with $M : \pR^n \to \R^{n\times n}$ defined by
\begin{equation}
\label{eq:defM}  M(y) =
  \begin{pmatrix}
    h_{1}''(y_1) & 0 & \hdots & 0 \\
    0 & h_2''(y_2) & \hdots & 0 \\
    \vdots & \vdots & \ddots & \vdots \\
    0 & 0 & \hdots & h_n''(y_n)
  \end{pmatrix}
  \begin{pmatrix}
    a_{11}(y) & a_{12}(y) & \hdots & a_{1n}(y) \\
    a_{21}(y) & a_{22}(y) & \hdots & a_{2n}(y) \\
    \vdots & \vdots & \ddots & \vdots \\
    a_{n1}(y) & a_{n2}(y) & \hdots & a_{nn}(y)
  \end{pmatrix}.
\end{equation}
Hence we have a positive dissipation \(I= - \dd H / \dd t\) if (the
symmetric part of) \(M\) is positive semi-definite. This motivates the
following definition, where the second part quantifies the dissipation.
\begin{defi}[Entropy structure]
  \label{def:entropy-structure}
  We say that the system~\eqref{eq:general-local-cross-diffusion} has
  an \emph{entropy structure} if there exist $n$ functions
  $h_1,\dots,h_n\in\fsC^0(\nnR)\cap\fsC^2(\pR)$ such that the
  corresponding matrix map $M:\pR^n\rightarrow \R^{n\times n}$ defined
  by \eqref{eq:defM} takes its values in the cone of positive definite
  matrices. We say that this entropy structure is \emph{uniform} when
  there exist furthermore $n$ functions
  $\alpha_1,\dots,\alpha_n:\nnR\rightarrow\nnR$ such that for
  all $(z,v) \in \R^n\times \pR^n$ it holds
  \begin{equation}
    \label{eq:quantised-m-positiviy-assumption}
    z^T\cdot M(v) z \ge \sum_{i=1}^{n} \alpha_i(v_i)^2 z_i^2.
  \end{equation}
  For a given entropy structure the functions $h_i$'s are called the
  \emph{entropy densities}, $\alpha_i$'s are the \emph{dissipations}
  and the functional \(H\) defined in
  \eqref{eq:general-additive-entropy} is called the \emph{entropy}.
\end{defi}
\begin{remark}
  From the assumed positive definiteness of the diffusion matrix
  \(A\), it directly follows that for every entropy structure all the
  functions \(h_i\), \(i=1,\dots,n\), are convex.
\end{remark}
\begin{remark}
  For typical examples, as the SKT system \eqref{eq:local-SKT} below,
  the entropy densities \(h_i\) have diverging derivative towards the
  origin so that we define the matrix map \(M\) only for positive
  arguments. In this work we also take \(\nnR\) for the range of the
  densities which is the most common case. In general the entropy
  structure can also be defined for bounded subsets of \(\R\),
  cf.~\cite{juengel-2015-boundedness-by-entropy}.
\end{remark}

Smooth solutions for the system
\eqref{eq:general-local-cross-diffusion} are known to exist, at least
locally in time, thanks to the work of
\textcite{amann-1990-quasilinear-parabolic-systems} which gives also a
criteria of explosion for such solutions. Apart from the very specific
case of triangular system \cite{HoanNguPha}, for global solutions the
current literature allows only weak solutions and relies crucially on
the entropy structure.

For an overview of such cross-diffusion systems we refer to
\textcite{juengel-2015-boundedness-by-entropy}, which gives a list of
examples in the introduction and also uses the quantified condition
\eqref{eq:quantised-m-positiviy-assumption}.  Note that
\cite{juengel-2015-boundedness-by-entropy} allows in principle more
general entropies but, apart from the volume-filling models, all
examples have the additive form \eqref{eq:general-additive-entropy}
required in this work.

\medskip

A guiding example is the SKT system with
densities \(u_1\) and \(u_2\)
\begin{equation}
  \label{eq:local-SKT}
  \left\{
    \begin{lgathered}
      \partial_t u_1 = \lap\big( (d_1+d_{12}u_2) u_1\big), \\
      \partial_t u_2 = \lap\big( (d_2+d_{21}u_1) u_2\big)
    \end{lgathered}
  \right.
\end{equation}
with parameters \(d_1,d_2,d_{12},d_{21} \ge 0\). This system has been
introduced by Shigesada, Kawasaki and
Teramoto~\cite{shigesada-kawasaki-teremoto-1979-interacting-species}. Writing
the system in divergence form
\eqref{eq:general-local-cross-diffusion}, the matrix
$(a_{ij}(u))_{ij}$ reads
\begin{align*}
  \begin{pmatrix}
    d_1+d_{12}u_2 & d_{12} u_1\\
    d_{21} u_2 &     d_2+d_{21}u_1
    \end{pmatrix}.
\end{align*}
For non-negative solutions this matrix has non-negative trace and
determinant. As remarked by \textcite{chenjun} the following entropy
allows to symmetrize the system
\begin{equation}
  \label{eq:two-species-entropy}
  H(u_1,u_2) := \int_{\T^d} \big( h_1(u_1) + h_2(u_2) \big),
\end{equation}
with
\begin{align}\label{eq:dens-ent-skt}
  h_1(z)=d_{21}\psi(z), \quad h_2(z):=d_{12}\psi(z), \quad \psi(z)=z\log(z)-z+1.
  \end{align}
Indeed, one checks that
\begin{align*}
  M(u_1,u_2) &=
  \begin{pmatrix}
    h_1''(u_1) & 0\\
    0 & h_2''(u_2)
  \end{pmatrix}
  \begin{pmatrix}
    d_1+d_{12}u_2 & d_{12} u_1\\
    d_{21} u_2 &     d_2+d_{21}u_1
    \end{pmatrix}
  = d_{12} d_{21}
  \begin{pmatrix}
    \star & 1 \\
    1& \star
  \end{pmatrix},
\end{align*}
so that $M$ is symmetric and still has non-negative determinant and
trace. Thus \(M\) is positive semi-definite and forms with \(H\) an
entropy structure again under the necessary condition that the solution is non-negative.

\medskip

It is also known (see \cite{juengel-2015-boundedness-by-entropy,
  desvillettes-lepoutre-moussa-trescases-2015-entropic-structure,
  lepoutre-moussa-2017-entropic-structure-duality-multiple-species}
for instance) that the previous entropy structure of the additive form
\eqref{eq:general-additive-entropy} can be found for substantial
generalization of \eqref{eq:local-SKT} in the following general class
of cross-diffusion systems
\begin{equation}
  \label{eq:local-SKT-rates-mu-kappa}
  \left\{
    \begin{lgathered}
      \partial_t u_1 = \lap\big( \mu_1(u_1,u_2)\, u_1\big), \\
      \partial_t u_2 = \lap\big( \mu_2(u_1,u_2)\, u_2\big),
    \end{lgathered}
  \right.
\end{equation}
where the non-linear functions $\mu_1$ and $\mu_2$ are assumed
$\fsC^0(\R^2_{\geq 0})\cap \fsC^1(\pR^2)$ so that
\eqref{eq:local-SKT-rates-mu-kappa} can be written in divergence form
\eqref{eq:general-local-cross-diffusion} in order for the entropy
structure to makes sense. For the analysis of the PDE the difficulty
comes from the cross-diffusion effect so that we will focus on the
case without self-diffusion (imposing that $\mu_i$ does not depend on
\(u_i\))
\begin{equation}
  \label{eq:local-SKT-rates-mu-kappa:noself}
  \left\{
    \begin{lgathered}
      \partial_t u_1 = \lap\big( \mu_1(u_2)\, u_1\big), \\
      \partial_t u_2 = \lap\big( \mu_2(u_1)\, u_2\big).
    \end{lgathered}
  \right.
\end{equation}

The contribution of this paper is a constructive answer to the
following question.
\begin{MQ}
  For a cross-diffusion system with an entropy \(H\) of the form
  \eqref{eq:general-additive-entropy}, does there exists a spatial
  mollification of the diffusion such that the mollified system still
  has an entropy?
\end{MQ}

These mollified systems are called \emph{nonlocal} because the
diffusion rate of one species at a given point $x$ does not depend
anymore solely on the population density at this place, but on a
\emph{space average} around it.

We provide a family of spatial mollifications of the cross-diffusion
keeping the entropy structure keeping the entropy structure, where our
intuition takes its origin from the article
\cite{daus-desvillettes-dietert-2019-entropic-structure} in which the
first author of the current article exhibited an entropy structure for
the SKT systems under a spatial discretization.  To the best of our
knowledge, the current literature does not offer any prior example of
persisting entropy structure for a nonlocal cross-diffusion systems.

\paragraph{Approximation results}

The usage of a spatial mollification was first proposed by
\textcite{ben-lep-mar-per,lepoutre-pierre-rolland-2012-relaxed-cross-diffusion}
where no rigorous link with the original model was established.

Having the entropy structure at hand, we can rigorously perform the
limit from the mollified nonlocal system to the original local
system. This gives immediately a new proof of existence of global weak
entropy solutions.

Such an existence result is non-trivial because an adequate
approximation scheme has to create non-negative solutions and to keep
the entropy structure, as it has been done before by an entropic
change of variable \cite{juengel-2015-boundedness-by-entropy} or with
a semi-discrete scheme
\cite{desvillettes-lepoutre-moussa-trescases-2015-entropic-structure}.

\paragraph{Derivation from particle models}

The main motivation comes from the derivation of many particle models
as a mean-field limit. The starting point is by
\textcite{fontbona-meleard-2014-non-local-cross-diffusion} who
performed a stochastic derivation of a regularised cross-diffusion
system with a nonlocal spatial regularisation. Their aim was not to
produce an adequate approximation scheme but to derive the SKT system
from a particle model. However, they could not handle the last step of
the derivation and they explicitly raised the question, whether it is
possible to find the classical local cross-diffusion system in the
limit of small regularisation.

A partial answer in this direction is given in
\cite{moussa-2020-triangular} in the special case of triangular
diffusion coefficients. The recent work by
\textcite{chen-daus-holzinger-juengel-2020-preprint-derivation} uses
the same approximation by spatial mollifiers and manage to prove
rigorously the limit of small regularization \emph{and} large
population at the same time. In the aspect of taking both asymptotic
limits at once, the analysis of
\cite{chen-daus-holzinger-juengel-2020-preprint-derivation} goes
beyond the program of
\textcite{fontbona-meleard-2014-non-local-cross-diffusion}. However,
for proving the uniform stability of the mollified systems,
\textcite{chen-daus-holzinger-juengel-2020-preprint-derivation} impose
the assumption of small cross-diffusion coefficients so that the
cross-diffusion terms can be handled perturbatively. Hence the result
of \cite{chen-daus-holzinger-juengel-2020-preprint-derivation} does
\emph{not} cover the full SKT system \eqref{eq:local-SKT}.

In general, our proposed regularisation scheme is different to the one
used in
\cite{fontbona-meleard-2014-non-local-cross-diffusion,chen-daus-holzinger-juengel-2020-preprint-derivation}
but agrees on the important example of a linear rate SKT system
\eqref{eq:local-SKT}. Hence we provide, to the best of our knowledge,
the first complete derivation of this popular cross-diffusion model \emph{via} non-local approximation.

For completeness, we note that other approaches for the derivation of
cross-diffusion systems are fast reaction asymptotics
\cite{iidadiff,ariane,DDJ} or spatial discretisations (without
convolution). The later method was formally proposed in
\cite{daus-desvillettes-dietert-2019-entropic-structure} and recently
revisited rigorously in \cite{ban-mou-mu}.

\paragraph{Plan}

In the following Subsections~\ref{sec:intro:torus}
and~\ref{sec:intro:general-regularisation-scheme}, we introduce our
nonlocal mollifications keeping the entropy structure. We then state
the existence and convergence results for the mollified systems in the
following Subsection~\ref{sec:intro:results}. In the remainder of the
paper these results are then proved.

As we focus on the approximation scheme, we will show existence of
solutions of the regularised systems by PDE techniques, where we
already see the effectiveness of the regularisation. We expect that
the stochastic derivation can be adapted; but leave a general
derivation from particle models for future work.

Another future direction is the study of the gradient flow
structure. Formally, the original local system often has a gradient
flow structure which in most studies is only used in the form of the
dissipation inequality (an exception is
\textcite{zinsl-matthes-2015-transport-distances}). Having found a
regularisation, we plan for future work to investigate the gradient
flow formulation of the nonlocal system and the limit towards the
local system. Such limits of gradient flows are an active field and we
only mention
\cite{serfaty-2011-gamma-convergence,braides-2014-local-minimization,mielke-2016-evolutionary-gamma-convergence}
as starting points.

\subsection{Regularisation on the torus}
\label{sec:intro:torus}

The starting point was
\cite{daus-desvillettes-dietert-2019-entropic-structure}, where the
entropy structure was understood for the linear rates SKT model in a
spatial discretisation. In this paper the intuition is to relate the
entropy structure to the reversibility of a Markov chain modelling an
\(N\)-particle system whose mean-field limit converges (formally) to
the spatially discrete system.

Briefly, the idea in
\cite{daus-desvillettes-dietert-2019-entropic-structure} is that, on a
particle model with discrete space variable, the entropy structure is
obtained by imposing that a pair of particles is jumping together with
a suitable rate. Trying to use this idea for a nonlocal approximation,
we intuitively want to make pairs of particle with a given distance
jump together. In order to identify the pairs, we therefore take the
convolution reflected between the two species.

For \(\domain = \T^d\) and a convolution kernel
\(\rho : \T^d \rightarrow \nnR\) this motivates the following
regularisation of \eqref{eq:local-SKT}
\begin{equation}
  \label{eq:regularised-SKT-simple}
  \left\{
    \begin{lgathered}
      \partial_t u_1 = \lap\big( (d_1+d_{12}\,u_2 \conv \rho) u_1\big) \\
      \partial_t u_2 = \lap\big( (d_2+d_{21}\,u_1 \conv \check{\rho}) u_2\big)
    \end{lgathered}
  \right.
\end{equation}
where \(\check{\rho}\) is the reflected convolution kernel, i.e.
\begin{equation*}
  \check{\rho}(y) = \rho(-y).
\end{equation*}
The key-observation is that, for any
function $\varphi$ one has formally
\begin{align*}
  \int_{\T^d} \varphi(u_2) \lap\big((u_1 \conv \check{\rho}) u_2)\big)
  &=\int_{\T^d} \varphi(u_2) \big\{(\Delta u_1\star \check{\rho})u_2 + 2 \nabla (u_1\star \check{\rho})\cdot \nabla u_2 + (u_1 \star\check{\rho})\Delta u_2 \big\}\\
  &=\int_{\T^d} \big[\rho\star(\varphi(u_2)u_2)\big] \Delta u_1 + 2 \big[\rho\star(\varphi(u_2)\nabla u_2)\big]\cdot \nabla u_1  + \big[\rho\star(\varphi(u_2)\Delta u_2)\big]u_1 \\
  &=\int_{\T^d}\int_{\T^d}\rho(y)\varphi(u_2(x-y))\lap_x( u_2(x-y)u_1(x))\,\dd y\,\dd x,
\end{align*}
which, using the translation operator $\tau_y u_2 = u_2(\cdot-y)$, summarises in
\begin{align*}
  \int_{\T^d} \varphi(u_2) \lap\big((u_1 \conv \check{\rho}) u_2)\big)
  = \int_{\T^d}\rho(y)\left\{\int_{\T^d} \varphi(\tau_y u_2) \lap\big( u_1 \tau_y u_2)\big)\right\}\,\dd y.
\end{align*}
In particular, for $H$ of the form \eqref{eq:general-additive-entropy} we have
\begin{equation*}
  \begin{split}
    - \frac{\dd}{\dd t}H(u_1,u_2)
    = \int_{\T^d} \rho(y) \left\{\int_{\T^d}
    \begin{pmatrix}
      \nabla u_1 \\
       \nabla \tau_y u_2
    \end{pmatrix}
    \cdot
    M\Big(u_1 ,\tau_y u_2\Big)
    \begin{pmatrix}
      \nabla u_1 \\
       \nabla \tau_y u_2
    \end{pmatrix}
\right\} \,\dd y,
  \end{split}
\end{equation*}
where \(M\) is the same matrix associated to $H$, as in the local
case. In particular, since the kernel is non-negative, the entropy
structure of the local case persists in the nonlocal system in the
sense that $H$ still defines a Lyapunov functional. The previous
computation can be adapted to the generalizations of the SKT system
\eqref{eq:local-SKT-rates-mu-kappa} with the following caution: the
spatial regularisation has to be applied after the nonlinearity,
without affecting the self-diffusion. We have more precisely the
following proposition.
\begin{prop}\label{prop:formal}
  Consider $\mu_1,\mu_2\in\fsC^0(\R^2_{\geq 0})\cap \fsC^1(\pR^2)$ and the corresponding system
  \eqref{eq:local-SKT-rates-mu-kappa}. If this system has an entropy
  $H$, then for any non-negative kernel $\rho$ of integral $1$, any
  solution of the following non-local system ($\tau_y$ is the
  translation operator)
  \begin{equation}
    \label{eq:regularised-general-skt-torus}
    \left\{
      \begin{lgathered}
        \partial_t u_1
        - \lap\left[ \int_{\T^d} \rho(y)\, \mu_1\big(u_1,\tau_y u_2)\big)\, \dd
          y\, u_1 \right]
        =0,
        \\
        \partial_t u_2
        - \lap\left[ \int_{\T^d} \check{\rho}(y)\, \mu_2\big(\tau_y u_1,u_2)\big) \,\dd y\, u_2  \right]
        =0
      \end{lgathered}
    \right.
  \end{equation}
  satisfies formally
  \begin{align*}
    \frac{\dd}{\dd t} H(u_1(t),u_2(t)) \leq 0.
  \end{align*}
  If the entropy structure of the system
  \eqref{eq:local-SKT-rates-mu-kappa} is furthermore assumed uniform
  with dissipation $\alpha_1$ and $\alpha_2$, then we have formally
  \begin{align*}
    \frac{\dd}{\dd t} H(u_1(t),u_2(t)) + D(t) \leq 0,
  \end{align*}
  where
  \begin{align*}
    D(t):= \int_{\T^d} \alpha_1(u_1(t))^2 |\nabla u_1(t)|^2 + \int_{\T^d} \alpha_2(u_2(t))^2 |\nabla u_2(t)|^2.
  \end{align*}
\end{prop}
\begin{proof}
  Denoting by $h_1$ and $h_2$ the entropy densities, we find by
  multiplying the first equation of
  \eqref{eq:regularised-general-skt-torus} by $h_1'(u_1)$ and
  integrating over $\T^d$ that
  \begin{align*}
    \frac{\dd}{\dd t} \int_{\T^d} h_1(u_1)\, \dd x
    &= -\int_{\T^d} h_1''(u_1)\nabla u_1 \cdot
      \left\{\int_{\T^d} \rho(y)\nabla( \mu_1(u_1,\tau_y u_2)u_1)\,\dd y\right\}
     \dd x\\
    &= -\int_{\T^d} \rho(y)
      \left\{ \int_{\T^d} h_1''(u_1)\nabla u_1
      \cdot \nabla(\mu_1(u_1,\tau_y u_2)u_1 )\, \dd x\right\}
      \,\dd y,
  \end{align*}
  where $\tau_y$ is the translation operator and $\nabla$ acts on the
  $x$ (not noted) variable only. We have a similar formula for the
  second equation, that is
  \begin{align*}
    \frac{\dd}{\dd t} \int_{\T^d} h_2(u_2)\, \dd x
    = -\int_{\T^d} \check{\rho}(y)
    \left\{ \int_{\T^d} h_1''(u_2)\nabla u_2
    \cdot \nabla( \mu_2(\tau_y u_1,u_2)u_2 )\, \dd x \right\}
    \,\dd y.
  \end{align*}
  Intuitively speaking, we want to collect the pairs \(u_1(x)\) and
  \(u_2(x-y)\) in both expressions. This motivates in the double
  integral in the variables \(x,y\) of the last r.h.s. the change of
  variable $(x,y)\mapsto (z-w,-w)$. Using that the translation
  commutes with differential operators, we find
  \begin{align*}
    \frac{\dd}{\dd t} \int_{\T^d} h_2(u_2)
    = -\int_{\T^d} \check{\rho}(-w)
    \left\{ \int_{\T^d} h_2''(\tau_w u_2)\nabla \tau_w u_2
    \cdot \nabla( \mu_1( u_1,\tau_w u_2)\tau_w u_2) \dd z\right\}
    \,\dd w.
  \end{align*}
  Since $\check{\rho}(-w)=\rho(w)$, renaming the variables as before,
  we collect both contributions as
  \begin{align*}
    \frac{\dd}{\dd t} H(u_1(t),u_2(t))= -\int_{\T^d} \rho(y)
    \left\{\int_{\T^d}
    \begin{pmatrix}
      \nabla u_1 \\
      \nabla \tau_y u_2
    \end{pmatrix}
    \cdot
    M\Big(u_1 ,\tau_y u_2\Big)
    \begin{pmatrix}
      \nabla u_1 \\
      \nabla \tau_y u_2
    \end{pmatrix}
    \right\} \,\dd y,
  \end{align*}
  where $M$ is given by \eqref{eq:defM}, the coefficients of the
  matrix $A$ being the one used to write
  \eqref{eq:local-SKT-rates-mu-kappa} in divergence form
  \eqref{eq:general-local-cross-diffusion}. The fact that $H$ is a
  Lyapunov functional and the precised dissipation in case of uniform
  entropy follow (for the latter, use the normalisation of $\rho$).
\end{proof}
If there is no self diffusion in the generalized system
\eqref{eq:local-SKT-rates-mu-kappa} like in
\eqref{eq:local-SKT-rates-mu-kappa:noself}, the nonlocal system
\eqref{eq:regularised-general-skt-torus} becomes simply
\begin{equation}
  \label{eq:regularised-general-skt-torus:noself}
  \left\{
    \begin{lgathered}
      \partial_t u_1 = \lap\big( (\mu_1(u_2)\star\rho) u_1\big), \\
      \partial_t u_2 = \lap\big( (\mu_2(u_1)\star\check{\rho}) u_2\big).
    \end{lgathered}
  \right.
\end{equation}
Thus, compared to
\textcite{fontbona-meleard-2014-non-local-cross-diffusion}, the
spatial regularisation is applied after the nonlinearity, while they
do it the opposite way in their stochastic derivation. However, in the
fundamental case of the (linear) SKT system \eqref{eq:local-SKT}, we
get the same system.

\medskip

For these systems with the Laplace structure another important role is
played by the duality estimates, see
\cite{desvillettes-lepoutre-moussa-trescases-2015-entropic-structure,
  lepoutre-moussa-2017-entropic-structure-duality-multiple-species,
  moussa-2020-triangular}. An advantage of the previous scheme is that
these duality estimates naturally continue to work in the nonlocal
versions.

\begin{remark}
  In the regularisation \eqref{eq:regularised-general-skt-torus} the
  rate \((\mu_i)_{i=1,2}\) is averaged with respect to the
  cross-diffusion influence but a possible nonlinear self-diffusion is
  not regularised. However, a nonlinear self-diffusion tends to
  improve the entropy-dissipation estimates and we thus focus on cases
  without self-diffusion.

  For stochastic derivations it is, nevertheless, interesting to also
  regularise the self-diffusion.  In a general setting this destroys
  the entropy structure and we need a compatibility with the entropy
  structure. For such a regularisation consider a symmetric kernel
  \(\sigma\), i.e.\ \(\check{\sigma}=\sigma\), and assume that
  \eqref{eq:local-SKT-rates-mu-kappa} can be written as
  \begin{equation*}
    \left\{
      \begin{lgathered}
        \partial_t u_1 = \lap\big( (\mu_1(u_2)+\kappa_1(u_1)) \, u_1\big), \\
        \partial_t u_2 = \lap\big( (\mu_2(u_1)+\kappa_2(u_2)) \, u_2\big),
      \end{lgathered}
    \right.
  \end{equation*}
  where the system without the \(\kappa\) has an entropy structure
  with an entropy \(H\) consisting of \(h_1\) and \(h_2\) and matrix
  map \(M\). We then propose the regularisation
  \begin{equation*}
    \left\{
      \begin{lgathered}
        \partial_t u_1(x)
        = \lap_x\left[
          \left(\int_{y\in\T^d} \rho(x{-}y)\, \mu_1\big(u_2(y)\big)
            \, \dd y
            +
            \int_{y\in\T^d} \sigma(x{-}y)\, \kappa_1\big(u_1(y)\big)\, \dd y
          \right)
          u_1(x) \right]
        \\
        \partial_t u_2(y)
        = \lap_y\left[
          \left(\int_{x\in\T^d} \rho(x{-}y)\, \mu_2\big(u_1(x)\big)
            \, \dd x
            +
            \int_{x\in\T^d} \sigma(x{-}y)\, \kappa_2\big(u_2(x)\big)\, \dd x
          \right)
          u_2(y) \right]
      \end{lgathered}
    \right.
  \end{equation*}
  For the dissipation we then find
  \begin{equation*}
    - \frac{\dd}{\dd t} H(u_1,u_2)
    = I_\mu + I_1 + I_2
  \end{equation*}
  where \(I_\mu\) is the dissipation with \(\kappa_1=\kappa_2=0\) and
  thus has a good sign. The new terms are after using symmetrisation
  \(\check{\sigma}=\sigma\)
  \begin{equation*}
    I_i = \frac 12
    \int_{x\in\domain} \int_{y\in\R^d} \sigma(y)
    \begin{pmatrix}
      \nabla u_i(x) \\ \nabla u_i(x-y)
    \end{pmatrix}
    \cdot
    N_i
    \begin{pmatrix}
      \nabla u_i(x) \\ \nabla u_i(x-y)
    \end{pmatrix}
    \dd y\, \dd x
  \end{equation*}
  with
  \begin{equation*}
    N_i =
    \begin{pmatrix}
      h_i''(u_i(x)) & 0 \\
      0 & h_i''(u_i(x-y))
    \end{pmatrix}
    \begin{pmatrix}
      \kappa_1(u_i(x-y))
      & u_i(x) \kappa_1'(u_i(x-y)) \\
      u_i(x-y) \kappa_1'(u_i(x))
      & \kappa_1(u_i(x))
    \end{pmatrix}
  \end{equation*}
  for \(i=1,2\).

  Hence \(H\) is still an entropy if \((N_i)_{i=1,2}\) are always
  positive semi-definite which gives an extra condition on the
  system. We note, however, that for the studied SKT
  system~\eqref{eq:local-SKT} this condition is always satisfied
  under the natural assumption that effect on the other species is of
  the same form as the self-diffusion effect, i.e. that it takes the
  form
  \begin{equation*}
    \left\{
      \begin{lgathered}
        \partial_t u_1 = \lap\big( (d_1+d_{12}u_2^{\alpha} + d_{11} u_1^{\beta}) u_1\big), \\
        \partial_t u_2 = \lap\big( (d_2+d_{21}u_1^{\beta} + d_{22} u_2^{\alpha}) u_2\big),
      \end{lgathered}
    \right.
  \end{equation*}
  for constants
  \(d_1,d_2,d_{11},d_{12},d_{21},d_{22},\alpha,\beta \in \pR\) with
  \(\alpha\beta \le 1\) (see, e.g.,
  \cite{desvillettes-lepoutre-moussa-trescases-2015-entropic-structure}
  for the discussion of the local case).
\end{remark}

\begin{remark}
  Consider the SKT system on a regular bounded set $\domain \in \R^d$
  with \emph{constant} Dirichlet boundary conditions\footnote{We need to assume constant boundary
    data in order to avoid boundary terms in the entropy estimate.};
  that is, for $x\in\partial\domain$ we impose for all times
  \(t\in\R_+\) $(u_1,u_2)(t,x)=(b_1,b_2) \in \pR^2$ and we will now
  explain how the previous regularisation scheme on the torus can be
  used to design an approximation procedure for this type of boundary
  conditions by a penalisation method.

  First note is that, whenever the system has an entropy structure,
  each elementary functions $h_i:\nnR\to\R$ can be exchanged with
  $z\mapsto h_i(z)+\ell_i(z)$ for any affine functions $\ell_i$,
  without changing the entropy estimate (which solely relies on second
  derivatives). W.l.o.g.\ we can therefore assume
  $h_i(b_i)=h_i'(b_i)=0$. By convexity $h_i$ reaches therefore its
  minimum at $b_i$.

  Now for an approximating system take $N$ large enough such that
  $\overline{\Omega}\csubset (-N,N)^d$ and identify the hypercube
  $[-N,N]^d$ with the flat torus $\mathbb{T}_N^d:=(\R/2N\mathbb{Z})^d$
  and consider for \(\ep > 0\) the approximating system
  \begin{align*}
    \partial_t u_{1,\ep} - \lap((d_1+d_{12}
    u_{2,\ep}\star\rho_\ep)u_{1,\ep})
    &=- \frac{1}{\ep}(u_{1,\ep}-b_1)\mathbf{1}_{\T^d_N\setminus\domain},\\
    \partial_t u_{2,\ep} - \Delta((d_2+d_{21}
    u_{1,\ep}\star\check{\rho}_\ep)u_{2,\ep})
    &= - \frac{1}{\ep}(u_{2,\ep}-b_2)\mathbf{1}_{\T^d_N\setminus\domain}.
  \end{align*}
  The key idea is that $h_i'(z)(z-b_i)$ is a non-negative function
  vanishing at only one point: multiplying respectively the equations
  by $h_i'(u_{i,\ep})$, we recover an entropy estimate with an extra
  penalisation term (which has the good sign). At the limit
  $\ep\rightarrow 0$ the penalisation forces therefore $u_i=b_i$
  outside $\domain$ and thus on $\partial\domain$.
\end{remark}

\subsection{General regularisation scheme}
\label{sec:intro:general-regularisation-scheme}

In the previous subsection we considered the special case of two
species on the torus. In this subsection we will generalise the
regularisation scheme to several species and general domains
\(\domain\). Here the boundary implies that
the specific Laplace structure as in
\eqref{eq:local-SKT-rates-mu-kappa} is not preserved and we have the
general divergence structure as in
\eqref{eq:general-local-cross-diffusion}, see
\cref{rem:laplace-general-regularisation}.

For the two densities case on the torus, we used the convolution in
order to define how a pair is interacting in the cross-diffusion. In
the general case of \(n\) densities on a domain \(\domain\), the
suitable generalisation is a kernel \(K: \domain^n \to \R_{\ge 0}\)
between all densities and the intuitive idea is that the
cross-diffusion between the densities
\(u_1(x_1),u_2(x_2),\dots, u_n(x_n)\) at positions
\(x_1,x_2,\dots,x_n \in \domain\) happens with the intensity
\(K(x_1,x_2,\dots,x_n)\). The idea of using a kernel on a bounded
domain has been proposed in
\cite{lepoutre-pierre-rolland-2012-relaxed-cross-diffusion}, where $K$
is the fundamental solution the (Neumann) operator
$\textnormal{Id}-\delta \lap$ with \(0 < \delta \ll 1\). However, the
authors kept the Laplace structure and applied the regularisation
before the nonlinearity so that the entropy structure was lost, see
\cref{rem:laplace-general-regularisation} below.

At the boundary such a general tuple cannot diffuse freely if we
impose no-flux boundary conditions. Hence in order to rule out
boundary terms we further assume that
\begin{equation}
  \label{eq:assumption-k-boundary}
  K(x_1,\dots,x_n) = 0
  \qquad\text{if }
  x_i \in \partial \domain
  \text{ for }
  i=1,\dots,n.
\end{equation}

A family of kernel \(K^\epsilon\) for \(\epsilon > 0\) then yields an
approximation of the local system if the kernel is concentrating on
the diagonal as \(\epsilon \to 0\), i.e.\ for a species
\(i=1,\dots,n\), a point \(x_i \in \domain\) and a sufficiently nice
test function \(\phi: \Omega^n \to \R\) it holds that
\begin{equation*}
  \prod_{j\not=i} \int_{x_j \in \Omega} \dd x_{j}
  K^\epsilon(x_1,\dots,x_n)\, \phi(x_1,\dots,x_n)
  \to \phi(x_i,\dots,x_i) \qquad
  \text{ as } \epsilon \to 0,
\end{equation*}
where we introduced the notation
\(\prod_{j\not=i} \int_{x_j \in \Omega} \dd x_{j}\) to denote the
repeated integral over all coordinates \(x_j\) with \(j\not = i\),
i.e.
\begin{equation*}
  \prod_{j\not=i} \int_{x_j \in \Omega} \dd x_{j}
  :=
  \int_{x_1 \in \Omega} \dd x_1 \dots
  \int_{x_{i-1} \in \Omega} \dd x_{i-1}
  \int_{x_{i+1} \in \Omega} \dd x_{i+1}
  \dots
  \int_{x_{n} \in \Omega} \dd x_{n}.
\end{equation*}

A natural candidate of such kernels \(K^\epsilon\) is a smoothing of
\begin{equation*}
  K^\epsilon(x_1,\dots,x_n) = C_{\epsilon} 1_{|x_i-x_j|\le \epsilon, i,j=1,\dots,n}
\end{equation*}
with a cutoff towards the boundary and a suitable constant \(C_{\epsilon}\).

We can now state our proposed general regularisation.
\begin{prop}\label{thm:formal-general}
  Let \(n \in \N\) be the number of densities and assume rates
  \((a_{ij})_{i,j=1,\dots,n}\) such that the local system
  \eqref{eq:general-local-cross-diffusion} has an entropy \(H\).

  For a constant \(\epsilon > 0\), a domain \(\domain \in \R^d\) and a
  kernel \(K : \domain^n \to \R_{\ge 0}\) satisfying
  \eqref{eq:assumption-k-boundary}, suppose of densities
  \(u_1,\dots,u_n : \domain \to \R_{\ge 0}\) evolving in time \(t\) by the
  nonlocal system (\(i=1,\dots,n\))
  \begin{equation}
    \label{eq:general-regularised-evolution}
    \begin{split}
      &\partial_t u_i(x_i)
      - \epsilon \lap u_i(x_i) \\
      &\quad- \divergence_{x_i}
      \left(
        \prod_{k\not = i} \int_{x_k \in \Omega} \dd x_k\,
        K(x_1,\dots,x_n)\,
        \sum_{j=1}^{n}
        a_{ij}\big(u_1(x_1),\dots,u_n(x_n)\big)
        \nabla u_j(x_j)
      \right) \\
      &= 0
    \end{split}
  \end{equation}
  supplemented in the case of boundaries with von Neumann boundary
  conditions
  \begin{equation}
    \label{eq:general-regularised-evolution-boundary}
    n \cdot \nabla u_i(x) = 0
    \text{ for }
    x \in \partial\domain.
  \end{equation}
  Then it holds formally that
  \begin{equation*}
    \frac{\dd}{\dd t} H(u_1(t),\dots,u_n(t))
    \le 0.
  \end{equation*}

  In the case of uniform dissipations \((\alpha_i)_i\) it holds that
  \begin{equation*}
    \frac{\dd}{\dd t} H(u_1(t),\dots,u_n(t)) + D(t)
    \le 0,
  \end{equation*}
  where
  \begin{equation*}
    D(t) =
    \sum_{i=1}^n
    \int_\domain
    \left[\epsilon\, h_i''(u_i(x)) + \alpha_i(u_i(x))^2 w_i(x) \right]\;
    |\nabla u_i(x)|^2\, \dd x
  \end{equation*}
  with the weights \((w_i)_{i=1,\dots,n}\) defined as
  \begin{equation}
    \label{eq:mass-assumption-k}
    \prod_{j\not=i} \int_{x_j \in \Omega} \dd x_{j}
    K(x_1,\dots,x_n) = w_i(x_i) \ge 0,\qquad
    \forall x_i \in \Omega.
  \end{equation}
\end{prop}

Here we added a small global diffusion with \(\epsilon\) in order to
compensate that the kernel \(K\) vanishes at the boundary so that the
system becomes uniformly parabolic and we can obtain global regularity
estimates for the regularised system.  By the
assumption~\eqref{eq:assumption-k-boundary} the imposed von Neumann
boundary conditions imply zero-flux boundary conditions.

\begin{proof}
  In order to obtain the estimate, the idea is to collect the
  interaction in a tuple \(u_1(x_1),\dots,u_n(x_n)\). We then find for
  the dissipation
  \begin{equation*}
    \begin{split}
      &- \frac{\dd}{\dd t} H(u(t)) \\
      &= - \sum_{i=1}^{n} \frac{\dd}{\dd t} \int_{x_i \in \domain}
      h_i(x_i)\, \dd x_i \\
      &= \sum_{i=1}^{n} \int_{x_i}
      \nabla u_i(x_i)\, h_i''(u_i(x_i)) \\
      &\qquad\quad\cdot
      \left[
        \epsilon +
        \prod_{k\not = i} \int_{x_k \in \Omega} \dd x_k\,
        K(x_1,\dots,x_n)\,
        \sum_{j=1}^{n}
        a_{ij}\big(u_1(x_1),\dots,u_n(x_n)\big)
        \nabla u_j(x_j)
      \right]\, \dd x_i \\
      &= \int_{x_1} \dd x_1 \dots \int_{x_n} \dd x_n
      K(x_1,\dots,x_n)
      \begin{pmatrix}
        \nabla u_1(x_1) \\
        \vdots \\
        \nabla u_n(x_n)
      \end{pmatrix}
      \cdot
      M\big(u_1(x_1),\dots,u_n(x_n)\big)
      \begin{pmatrix}
        \nabla u_1(x_1) \\
        \vdots \\
        \nabla u_n(x_n)
      \end{pmatrix} \\
      &\qquad + \epsilon
      \int_{x \in \domain} \sum_{i=1}^{n}
      h_i''(u_i(x))\, |\nabla u_i(x)|^2\, \dd x,
    \end{split}
  \end{equation*}
  where \(M\) is the matrix from the entropy structure,
  \cref{def:entropy-structure}, and the boundary terms vanish due to
  the von Neumann boundary condition and
  \eqref{eq:assumption-k-boundary}.

  By the assumed sign of the matrix \(M\) and the lower bound by
  \(\alpha_i\), respectively, the result follows.
\end{proof}

\begin{remark}
  \label{rem:laplace-general-regularisation}

  The previous regularisation \eqref{eq:regularised-SKT-simple} for two
  species in the simple setting \(\domain=\R^d\) or \(\domain=\T^d\) is
  exactly recovered by setting \(K(x_1,x_2) = \rho(x_1-x_2)\) and
  dropping the normal diffusion with \(\epsilon\).

  This leaves the question whether the Laplace structure of a system of
  a the form \eqref{eq:local-SKT-rates-mu-kappa} can be preserved in the
  nonlocal version. Applying the regularisation procedure for a general
  kernel \(K\), we can rewrite the regularised evolution in the Laplace
  structure if
  \begin{equation*}
    \nabla_x K(x,y) = - \nabla_y K(x,y).
  \end{equation*}
  This, however, is only true if \(K\) has a convolution structure and
  thus does not work for domains with boundaries. Indeed we find for the
  two species system \eqref{eq:local-SKT-rates-mu-kappa} the
  regularisation
  \begin{equation}
    \label{eq:regularised-skt-general-k}
    \left\{
      \begin{lgathered}
        \begin{split}
          \partial_t u_1(x)
          &- \epsilon \lap u_1 - \lap_x\left[ \int_{y\in\domain} K(x,y)\, \mu_1\big(u_1(x),u_2(y)\big)\, \dd
            y\, u_1(x) \right] \\
          &=- \nabla_x
          \left[
            \int_{y\in\domain}
            \big[
            (\partial_x K)(x,y) + (\partial_y K)(x,y)
            \big]
            \mu_1\big(u_1(x),u_2(y)\big)\, \dd y\, u_1(x)
          \right],
        \end{split}
        \\
        \begin{split}
          \partial_t u_2(y)
          &- \epsilon \lap u_2 - \lap_y\left[ \int_{x\in\domain} K(x,y)\, \mu_2\big(u_1(x),u_2(y)\big)\, \dd
            x\, u_2(y) \right] \\
          &= - \nabla_y
          \left[
            \int_{x\in\domain}
            \big[
            (\partial_x K)(x,y) + (\partial_y K)(x,y)
            \big]
            \mu_2\big(u_1(x),u_2(y)\big)\, \dd x\, u_2(y)
          \right],
        \end{split}
      \end{lgathered}
    \right.
  \end{equation}
  which contains corrector terms for the defect of the convolution
  structure on the RHS.

  For the linear rate SKT system, this matches the regularisation
  following the discrete structure in
  \cite{daus-desvillettes-dietert-2019-entropic-structure}, where we
  identified the entropy with the reversibility of a corresponding
  Markov chain, see \cref{sec:reversibility}.
\end{remark}

\subsection{Results}
\label{sec:intro:results}

Having introduced the regularisation schemes, we can now state our
rigorous existence and approximation results.

Our first result shows that the regularisation
\eqref{eq:regularised-general-skt-torus:noself} for
\eqref{eq:local-SKT-rates-mu-kappa:noself} is sufficient to find
solutions satisfying the entropy-dissipation inequality. It will be
clear from the proof below that the diffusivity $\mu_i$'s could be
assumed sublinear, instead of being controlled by the entropy
densities. The self-diffusion could be included \emph{via} the more
general approximation \eqref{eq:regularised-general-skt-torus}
(for which there is a similar existence result) but we have chosen to
avoid it to simplify the presentation.
\begin{thm}
  \label{thm:existince-general-SKT}
  Consider the generalised SKT
  system~\eqref{eq:local-SKT-rates-mu-kappa:noself} with
  $\mu_1,\mu_2\in \fsC^0(\R_{\geq 0})\cap \fsC^1(\pR)$. Assume that it
  admits a uniform entropy structure with entropy $H$, entropy
  densities \(h_1,h_2 \in \fsC^0(\nnR)\cap\fsC^2(\pR)\) and
  dissipations \(\alpha_1,\alpha_2 : \nnR \to \nnR\). Assume
  furthermore two positive constants $\delta,\textnormal{A}$ such that
  for all \(z \in \nnR\)
  \begin{align}\label{ineq:assmu}
    \delta \leq \mu_1(z) \leq \textnormal{A}(1+h_2(z))
    \quad\text{and}\quad
    \delta \leq \mu_2(z) \leq \textnormal{A}(1+h_1(z)).
  \end{align}
  Fix $\rho\in\fsC^2(\T^d)$ non-negative having integral $1$ over
  $\T^d$, and a bounded initial data \(u_1^\init,u_2^\init\)
  satisfying for some positive constant $\gamma$
  \begin{align*}
    \gamma\leq u_i^\init \leq \gamma^{-1},
  \end{align*}
  so that $H^\init:=H(u_1^\init,u_2^\init) < +\infty$. Then, there
  exist positive functions
  \begin{align}\label{reg:ui}
    u_1,u_2 \in \fsC^0([0,T];\fsL^2(\T^d))\cap \fsL^2(0,T;\H^1(\T^d))\cap\fsL^\infty(0,T;\fsL^\infty(\T^d)),
  \end{align}
  such that $(u_1,u_2)$ is a distributional solution to the system
  \eqref{eq:regularised-general-skt-torus:noself} initiated by
  $(u_1^{\init},u_2^{\init})$. This solution $(u_1,u_2)$ satisfies
  furthermore the following estimates for $i=1,2$:
  \begin{itemize}
  \item \emph{conservation of the mass:}
    $u_i\in\fsC^0([0,T];\fsL^1(\T^d))$ and for $t\in [0,T]$
    \begin{align}\label{ineq:mass}
      \int_{\T^d} u_i(t) = \int_{\T^d} u_i^\init.
    \end{align}
  \item \emph{entropy estimate:} $h_i(u_i)\in\fsC^0([0,T];\fsL^1(\T^d))$
    and for $t\in [0,T]$
    \begin{align}\label{ineq:ent}
      H(u_1(t),u_2(t)) + \int_0^t D(s)\,\dd s \leq H^\init,
    \end{align}
    where
    \begin{align*}
      D(t) &:= \int_{\T^d} \alpha_1(u_1(t))^2|\nabla u_1(t)|^2 + \alpha_2(u_2(t))^2|\nabla u_2(t)|^2.
    \end{align*}
  \item \emph{maximum principle:}
    \begin{align}\label{ineq:max}
        \gamma \exp\left(-\textnormal{A}\textnormal{B}_{T,\init}\|\Delta \rho\|_{\fsL^\infty(\T^d)}\right) \leq u_i \leq \gamma^{-1}\exp\left(\textnormal{A}\textnormal{B}_{T,\init}\|\Delta \rho\|_{\fsL^\infty(\T^d)}\right),
    \end{align}
    where
    \[\textnormal{B}_{T,\init}:=T(1+ H^\init).\]
  \item \emph{duality estimate:}
    \begin{multline}\label{ineq:dua}
      \int_{Q_T} \big(\big[\mu_1(u_2)\star\rho\big] u_1
      + \big[\mu_2(u_1)\star\check{\rho}\big] u_2\big)\;(u_1+u_2) \\
      \lesssim_{d} (1+2\textnormal{A}\textnormal{B}_{T,\init}) \left(\int_{\T^d} (u_1^\init)^2 + \int_{\T^d} (u_2^\init)^2\right),
    \end{multline}
    where the constant behind $\lesssim_{d}$ depends only on the
    dimension \(d\).
  \end{itemize}
\end{thm}
\begin{remark}
  The upper-bound in assumption \eqref{ineq:assmu} is natural for many
  cross-diffusion systems. For instance if $\mu_1$ and $\mu_2$ are
  given by power-laws (as in
  \cite{desvillettes-lepoutre-moussa-trescases-2015-entropic-structure}),
  the entropy densities are precisely given by the same exponents
  (with an exception for the linear case). See also \cref{rem:sktasym}.
\end{remark}

Using the uniform control by the entropy, we can prove the following
limit theorem. Together with the previous existence result, this
shows, as a by by-product, the (known) existence of weak solutions to
the generalized SKT system \eqref{eq:local-SKT-rates-mu-kappa:noself}.
\begin{thm}\label{thm:sktasym}
  Consider the assumptions of Theorem~\ref{thm:existince-general-SKT},
  for a sequence of non-negative functions
  $(\rho_n)_n \in\fsC^2(\T^d)^{\mathbb{N}}$ which converges weakly
  towards the Dirac mass, with dissipation rates $\alpha_1$ and
  $\alpha_2$ vanishing on a set of measure $0$. Assume furthermore
  that the diffusivities are strictly subquadratic or controlled by
  the entropy densities, that is
  \begin{align}\label{lim:mubound}\lim_{z\rightarrow +\infty}
    \frac{\mu_1(z)}{h_2(z)+ z^2}+ \frac{\mu_2(z)}{h_1(z)+ z^2} =
    0.
  \end{align}
  Then, the corresponding sequence of solutions
  $(u_{1,n},u_{2,n})_n$ given by
  Theorem~\ref{thm:existince-general-SKT} converges (up to a
  subsequence) in $\fsL^1(Q_T)$ towards a weak global solution
  $(u_1,u_2)$ of the SKT system which satisfies for a.e. $t\in[0,T]$
  the conservation of the mass \eqref{ineq:mass}, the entropy
  estimate \eqref{ineq:ent} and the following duality estimate
  \begin{multline}\label{ineq:dua:good}
    \int_{Q_T} \big(\mu_1(u_2) u_1 + \mu_2(u_1)u_2\big)(u_1+u_2)\\
    \lesssim_{d} (1+2 \textnormal{AB}_{T,\init}) \int_{\T^d} (u_1^\init)^2 + \int_{\T^d} (u_2^\init)^2.
  \end{multline}
\end{thm}
\begin{remark}\label{rem:sktasym}
  The assumption \eqref{lim:mubound} is crucial to avoid any
  concentration in the non-linearities of the system. However, in
  practice (see for instance the power-law case in
  \cite{desvillettes-lepoutre-moussa-trescases-2015-entropic-structure})
  the control of gradients of the entropy estimate gives raise (by
  Sobolev embedding) to another estimate on $\mu_1(u_2)$ and
  $\mu_2(u_1)$.
\end{remark}

In a general setting we described the regularisation scheme
\eqref{eq:general-regularised-evolution}, for which we can state the
following existence result.
\begin{thm}
  \label{thm:general-system-existence-regularised}

  Consider a cross-diffusion system
  \eqref{eq:general-local-cross-diffusion} for \(n\) species and rates
  \(a_{ij} \in \fsC^0(\R_{\ge 0}^n)\), \(i,j=1,\dots,n\). Assume that
  it admits a uniform entropy structure with entropy $H$, entropy
  densities
  \(h_1,\dots,h_n\in \mathscr{C}^0(\R_{\geq
    0})\cap\mathscr{C}^2(\R_{>0})\) and dissipations
  \(\alpha_1,\dots,\alpha_n : \R_{\ge 0} \to \R_{\ge 0}\).

  For \(i\not = j\) define \(\tilde{a}_{ij} : \R_{\ge 0}^n \to \R\) by
  \begin{equation*}
    \partial_j \tilde{a}_{ij}(v) = a_{ij}(v), \quad v \in \R_{\ge 0}^n,
    \qquad\text{and}\qquad
    \tilde{a}_{ij}(v) = 0, \quad \text{if } v_j=0.
  \end{equation*}
  Suppose that \(v \mapsto \tilde{a}_{ij}(v)\) is continuously
  differentiable with respect to \(v_i\) and that there exists a
  constant \(A\) such that for all \(v \in \R_{\ge 0}^n\) and
  \(i=1,\dots,n\)
  \begin{equation*}
    a_{ii}(v) \le A\, (1 + h_1(v_1) + \dots + h_n(v_n)),
  \end{equation*}
  \begin{equation*}
    \frac{\tilde{a}_{ij}(v)}{v_i} \le A\, (1 + h_1(v_1) + \dots + h_n(v_n)),
  \end{equation*}
  and
  \begin{equation*}
    \partial_i \tilde{a}_{ij}(v) \le A\, (1 + h_1(v_1) + \dots + h_n(v_n)).
  \end{equation*}

  Let \(\domain \in \R^d\) be a domain with piecewise \(\fsC^1\)
  boundary and \(K \in \fsC^2_c(\domain^n)\) be a nonnegative kernel
  satisfying \eqref{eq:assumption-k-boundary}. Further fix bounded
  initial data \(u^\init = (u_1^\init,\dots,u_n^\init)\) satisfying
  for some positive constant $\gamma$
  \begin{align*}
    \gamma\leq u_i^\init \leq \gamma^{-1},
  \end{align*}
  so that $H^\init:=H(u^\init) < +\infty$. Then, there exists positive functions
   \begin{align}\label{eq:general:regularity}
      u_1,\dots,u_n \in \fsC^0([0,T];\fsL^2(\T^d))\cap \fsL^2(0,T;\H^1(\T^d))\cap\fsL^\infty(0,T;\fsL^\infty(\T^d)),
   \end{align}
   such that $(u_1,\dots,u_n)$ is a distributional solution to the
   system \eqref{eq:regularised-skt-general-k} with initial data
   \(u^\init\) and von Neumann boundary data
   \eqref{eq:general-regularised-evolution-boundary}. Furthermore, the
   solution $u=(u_1,\dots,u_n)$ satisfies furthermore the following
   estimates for $i=1,\dots,n$:
   \begin{itemize}
   \item \emph{conservation of the mass:}
     $u_i\in\fsC^0([0,T];\fsL^1(\domain))$ and for $t\in [0,T]$
     \begin{align}\label{eq:approx-general:mass}
       \int_{\T^d} u_i(t) = \int_{\T^d} u_i^\init.
     \end{align}
   \item \emph{entropy estimate:}
     $h_i(u_i)\in\mathscr{C}^0([0,T];\fsL^1(\domain))$ and for
     $t\in [0,T]$
     \begin{align}\label{eq:approx-general:entropy}
       H(u(t))) + \int_0^t D(s)\,\dd s \leq H^\init,
     \end{align}
     \begin{equation*}
       D(t) =
       \sum_{i=1}^n
       \int_\domain
       \left[\epsilon\, h_i''(u_i(x)) + \alpha_i(u_i(x))^2 w_i(x) \right]\;
       |\nabla u_i(x)|^2\, \dd x
     \end{equation*}
     with the weights \((w_i)_{i=1,\dots,n}\) defined in
     \eqref{eq:mass-assumption-k}.
   \item \emph{maximum principle:}
     \begin{align}\label{eq:approx-general:max}
       \gamma
       \exp\left(-MT\right)
       \leq u_i
       \leq \gamma^{-1}\exp\left(MT\right),
     \end{align}
     where
     \begin{equation}
       \label{eq:growth-constant-m}
       M = 2A\, \max(\|K\|_\infty,\|\nabla K\|_\infty,\|\nabla^2 K\|_\infty)\,
       (|\domain| + H^\init).
     \end{equation}
   \item \emph{\(\epsilon\) regularity:}
     \begin{equation*}
       \sup_{t\in [0,T]} \| u_i(t,\cdot) \|_{\fsL^2(\domain)}^2
       + \epsilon \int_0^T \| \nabla u_i(t,\cdot) \|_{\fsL^2(\domain)}^2
       \le
       \exp\left[TM
         \Big(
         2+\frac{1}{\epsilon}
         \Big)
       \right]
       \| u_i^\init \|_{\fsL^2(\domain)}^2.
     \end{equation*}
   \end{itemize}
\end{thm}

Under the assumption that the dissipation is big enough, one can
conclude that the approximations converge to the local version. In the
setting of their time-discretisation approximation scheme,
\cite{chen-daus-juengel-2017-global-existence} discusses possible
conditions for such a convergence. Nevertheless, they need to treat
the SKT case separately.

As the SKT case is the motivating example, we focus on the SKT case,
where we replace the duality estimate with a positive
self-diffusion. For \(n\) species with densities \(u=(u_1,\dots,u_n)\)
the SKT system corresponds to the evolution
\begin{equation}
  \label{eq:skt-n-species}
  \partial_t u_i = \lap \left(d_i u_i + \sum_{j=1}^{n} d_{ij} u_j u_i \right)
\end{equation}
with constants \(d_1,\dots,d_n\ge 0\) and \((d_{ij})_{ij} \ge
0\). Furthermore, suppose that there exist weights
\(\pi_1,\dots,\pi_n \ge 0\) such that the diffusion coefficients
satisfy the detailed balance condition
\begin{equation}
  \label{eq:skt-detailed-balance}
  \pi_i d_{ij} = \pi_j d_{ji}, \qquad
  \text{for } i,j=1,\dots,n,
\end{equation}
see \cite{daus-desvillettes-dietert-2019-entropic-structure} for a
discussion on the condition. Then the evolution
\eqref{eq:skt-n-species} has an entropy structure with
\begin{equation*}
  h_i(z) = \pi_i\, \big( z \log z - z + 1\big)
\end{equation*}
and dissipation
\begin{equation*}
  \alpha_i(z) = \pi_i d_{ii}.
\end{equation*}

\begin{thm}\label{thm:general-skt-asymptotic}
  Given a bounded domain \(\domain \subset \R^d\) with \(\fsC^1\)
  boundary and an increasing sequence of sets \((A_m)_{m\in\N}\) with
  \(A_m \csubset \domain\) and \(A_m \uparrow \Omega\) as
  \(m\to\infty\). Suppose that there exists a constant \(c\) and
  extension operators \(\sExt_m : W^{1,2}(A_m) \to W^{1,2}(\R^d)\)
  such that \(\|\sExt_m\|_{W^{1,2}\to W^{1,2}} \le c\) and
  \(\|\sExt_m\|_{L^p \to L^p} \le c\) for \(p=2+2/d\).

  Assume a corresponding sequence of non-negative regularisation
  kernels \((K^m)_{m\in\N}\) in \(\fsC^2_c(\domain^n)\) for \(n\)
  densities satisfying \eqref{eq:assumption-k-boundary} with weights
  \(w^m_i\), \(i=1,\dots,n\), as in
  \eqref{eq:mass-assumption-k}. Suppose that the weights always map to
  \([0,1]\) and
  \begin{equation*}
    w^m_i(x) = 1,\qquad
    \text{for } x \in A_m
    \text{ and }
    i=1,\dots,n.
  \end{equation*}
  Moreover, suppose that \(K^m\) concentrates along the diagonal,
  i.e.\
  \begin{equation*}
    K^m(x_1,\dots,x_n) = 0,
    \quad
    \text{if }
    |x_i-x_j| \ge \frac 1m
    \text{ for some }
    i,j=1,\dots,n.
  \end{equation*}

  Consider the SKT system \eqref{eq:skt-n-species} for \(n\) densities
  with constants \(d_1,\dots,d_n\ge 0\) and \((d_{ij})_{ij} \ge 0\)
  and weights \(\pi_1,\dots,\pi_n \ge 0\) satisfying
  \eqref{eq:skt-detailed-balance} and \(d_{ii} > 0\) for
  \(i=1,\dots,n\) with initial data
  \(u^\init = (u_1^\init,\dots,u_n^\init)\) with
  \begin{equation*}
    \gamma \le u_i^\init \le \gamma^{-1}
  \end{equation*}
  for \(i=1,\dots,n\) and a constant \(\gamma \in \pR\).

  Then there exists a sequence of \((\epsilon_m)_{m\in\N}\) with
  \(\epsilon_m \downarrow 0\) as \(m \to \infty\) such that the
  approximating solutions \((u^m)_m\) as constructed in
  \cref{thm:general-system-existence-regularised} converge along a
  subsequence to \(u\) in \(\fsL^q([0,T)\times \domain; \nnR^n)\) with
  \(q=2+(1/2d)\). The limit \(u\) is a non-negative weak solution to
  \eqref{eq:skt-n-species} with the no-flux boundary conditions
  satisfying the entropy-dissipation inequality, i.e.\ for
  \(\phi \in \fsC^\infty([0,T] \times \domain)\) with
  \(\phi(T,\cdot) \equiv 0\) and \(i=1,\dots,n\) it holds
  \begin{equation*}
    - \int_0^T \int_\domain u_i\, \partial_t \phi
    + \int_0^T \int_\domain
    \left(\sum_{j=1}^{n} a_{ij}(u) \nabla u_j\right) \cdot
    \nabla \phi
    = \int_{\domain} u_i^\init \phi(0,\cdot),
  \end{equation*}
  where
  \begin{equation*}
    a_{ij}(u) =
    \begin{cases}
      d_i + 2 d_{ii} u_i + \sum_{j \not = i} d_{ij} u_j & \text{if } i=j, \\
      d_{ij} u_i & \text{otherwise}.
    \end{cases}
  \end{equation*}
\end{thm}
\begin{remark}
  A sequence of such sets \(A_m\) can be constructed for locally
  Lipschitz domains. For this, locally write the boundary as a graph
  of \(d-1\) variables and locally then such a sequence can be
  constructed. For the construction of extensions we refer to the
  treatment of \cite[Section 5.4]{evans-2010-pde-book} and
  \cite[Section VI]{stein-1970-singular-integrals}.
\end{remark}

\section{The convolution scheme on the flat torus}\label{sec:convo}

This section is dedicated to the proofs of
\cref{thm:existince-general-SKT} and \cref{thm:sktasym} on the
torus. For the domain, we introduce the notation
\begin{equation*}
  Q_T := [0,T) \times \T^d,
\end{equation*}
and start by recalling some useful results about the Kolmogorov
equation, that is
\begin{align}
\label{eq:kol1}  \partial_t z -\Delta(\mu z) &= G,\\
\label{eq:kol2}  z(0,\cdot) &= z^\init,
\end{align}
where $G$, $\mu$ and $z^\init$ are given and $z$ is the
unknown. Solutions will be understood in the following sense:
\begin{defi}\label{defi:kol}
  Given a measurable function $\mu:Q_T\rightarrow\R$, and two
  square-integrable functions $G,z^\init:Q_T\rightarrow\R$ we say that
  $z\in\fsL^1(Q_T)$ is a distributional solution of \eqref{eq:kol1} --
  \eqref{eq:kol2} if $z\mu$ is integrable on $Q_T$ and for all test
  function $\varphi\in\mathscr{D}(Q_T)$ it holds that
  \begin{align*}
    -\int_{Q_T} z(\partial_t \varphi+\mu\Delta\varphi)
    = \int_{\T^d} z^\init \varphi(0,\cdot) + \int_{Q_T} G\varphi.
  \end{align*}
\end{defi}

\subsection{Reminder on the Kolmogorov equation}

The following result is directly extracted from
\cite{moussa-2020-triangular}, more precisely merging results obtained
in Theorem~3, Proposition~2 and Proposition~3 therein.
\begin{thm}\label{thm:kol}
  Fix $\mu\in\fsL^\infty(Q_T)$ such that $\inf_{Q_T}\mu>0$. For any
  $z^\init\in\fsL^2(\T^d)$ there exists a unique solution $z$ to
  \eqref{eq:kol1} -- \eqref{eq:kol2} in the sense of
  Definition~\ref{defi:kol}. This solution belongs to $\fsL^2(Q_T)$ and
  satisfies
  \begin{itemize}
  \item \emph{maximum principle}: if $G$ and $z^\init$ are non-negative, then so is $z$;
  \item \emph{duality estimate}: $\mu^{1/2}z \in \fsL^2(Q_T)$ and
    \begin{align*}
      \int_{Q_T}  \mu z^2 \lesssim_d \left(1+\int_{Q_T}\mu\right)\left(\int_{\T^d} (z^\init)^2 + T\int_{Q_T} G^2\right),
    \end{align*}
    where the constant behind $\lesssim_d$ depends only on the dimension;
  \item \emph{sequential stability}: for fixed $G$ and $z^\init$ as
    above, the map $\mu \mapsto z$, restricted to those $\mu$ who are
    bounded and positively lower-bounded, is continuous in the
    $\fsL^1(Q_T)$ topology for the argument \(\mu\) and the $\fsL^2(Q_T)$
    topology for the image \(z\).
  \end{itemize}
\end{thm}
We will use two corollaries of the previous theorem.
\begin{cor}\label{cor:kol1}
  Consider the assumptions of \cref{thm:kol}, with $G=0$. If
  furthermore $z^\init$ is bounded with
  $\gamma\leq z^\init\leq \gamma^{-1}$ for some positive constant
  $\gamma$ and if $\Delta \mu\in\fsL^1(0,T;\fsL^\infty(\T^d))$, then
  $z\in\fsL^\infty(Q_T)$ with the estimate
  \begin{align*}
    \gamma \exp\left(-\int_0^T \|(\Delta \mu)^-(s)\|_{\fsL^\infty(\T^d)}\dd s\right)
    \leq z \leq
    \gamma^{-1}\exp\left(\int_0^T\|(\Delta \mu)^+(s)\|_{\fsL^\infty(\T^d)}\,\dd s\right),
  \end{align*}
  where the exponents ${}^+$ and ${}^-$ refer to (respectively) the
  positive and negative parts.
\end{cor}
\begin{proof}
  Define
  \begin{align*}
    \Phi(t) &:= \gamma^{-1} \exp\left(\int_0^t \|(\Delta \mu)^+(s)\|_{\fsL^\infty(\T^d)}\,\dd s\right)-z,\\
    \Psi(t) &:= z-\gamma \exp\left(-\int_0^t \|(\Delta \mu)^-(s)\|_{\fsL^\infty(\T^d)}\,\dd s\right),
  \end{align*}
  which satisfy
  \begin{align*}
    (\partial_t \Phi -\Delta(\mu \Phi))(t,x) &= (\Phi+z)(t,x)(\|(\Delta \mu)^+(t)\|_{\fsL^\infty(\T^d)}-\Delta \mu(t,x))\geq 0, \\
    (\partial_t \Psi -\Delta(\mu \Psi))(t,x) &= (z-\Psi)(t,x)(\Delta \mu(t,x) + \|(\Delta \mu)^-(t)\|_{\fsL^\infty(\T^d)})\geq 0.
  \end{align*}
  The conclusion follows using the maximum principle of
  Theorem~\ref{thm:kol}, since $\Phi$ and $\Psi$ are initially
  non-negative.
\end{proof}
\begin{cor}\label{cor:kol2}
  Consider the assumptions of Theorem~\ref{thm:kol}, with $G=0$. If
  furthermore $\Delta \mu\in\fsL^1(0,T;\fsL^\infty(Q_T))$, then
  $z\in\fsL^\infty(0,T;\fsL^2(\T^d))\cap\fsL^2(0,T;\H^1(\T^d))$ with the
  following estimate for a.e. $t\in[0,T]$
  \begin{align}\label{ineq:cor2}
    \int_{\T^d} z(t)^2 +\int_0^t \int_{\T^d} \mu |\nabla z|^2  \leq \exp \left(\int_0^t \|(\Delta \mu)^+(s)\|_{\fsL^\infty(\T^d)}\,\dd s \right) \int_{\T^d} (z^\init)^2.
  \end{align}
\end{cor}
\begin{proof}
  We first assume that $\mu$ and the initial data are smooth. In that
  case, we can rewrite the Kolmogorov equation \eqref{eq:kol1} as
  standard parabolic equation, and we get the smoothness of the
  solution $z$. In this situation, we can rigorously multiply the
  equation by $z$ and integrating by parts, to get
  \begin{align*}
    \frac12 \frac{\dd}{\dd t} \int_{\T^d} z(t)^2 + \int_{\T^d}
    \mu(t)|\nabla z(t)|^2
    &= - \int_{\T^d} z(t) \nabla z(t) \cdot \nabla \mu(t) \\
    &= \frac12 \int_{\T^d} z(t)^2 \Delta \mu(t)
      \leq \frac12 \|(\Delta \mu)^+(t)\|_{\fsL^\infty(\T^d)}\int_{\T^d} z(t)^2.
  \end{align*}
  We have thus
  \begin{multline*}
    \frac12 \frac{\dd}{\dd t} \left\{\exp\left(-\int_0^t \|(\Delta \mu)^+(s)\|_{\fsL^\infty(\T^d)}\,\dd s\right)\int_{\T^d} z(t)^2\right\}\\ + \exp\left(-\int_0^t \|(\Delta \mu)^+(s)\|_{\fsL^\infty(\T^d)}\,\dd s  \right)\int_{\T^d} \mu(t)|\nabla z(t)|^2\leq 0,
  \end{multline*}
  and we infer after time integration the stated estimate. For the
  moment, we only established the estimate in the case of smooth
  data. Replacing $\mu$ and $z^\init$ by smooth approximations
  $(\mu_n)_n$ and $(z^\init_n)_n$, approaching them in $\fsL^1(Q_T)$ and
  $\fsL^2(\T^d)$ respectively, with furthermore
  $\|(\Delta \mu_n)^+\|_{\fsL^\infty(Q_T)} \leq \|(\Delta
  \mu)^+\|_{\fsL^\infty(Q_T)}$, we get a sequence $(z_n)_n$ which, by
  the sequential stability of Theorem~\ref{thm:kol}, approaches $z$ in
  $\fsL^2(Q_T)$. The usual semi-continuity argument for weak convergence
  allows to obtain that
  $z\in \fsL^\infty(0,T;\fsL^2(\T^d))\cap\fsL^2(0,T;\H^1(\T^d))$, with the
  estimate \eqref{ineq:cor2} being satisfied for a.e. $t$.
\end{proof}
\subsection{Proof of \cref{thm:existince-general-SKT}}\label{subsec:exgenSKT}
\begin{proof}
  We start by proving the four \emph{a priori} estimates, under the
  assumption of positivity and regularity \eqref{reg:ui}.
  \begin{itemize}
  \item \textsf{\emph{conservation of the mass:}} since
    $u_i\in\fsC^0([0,T];\fsL^2(\T^d))$, it also belongs to
    $\fsC^0([0,T];\fsL^1(\T^d))$ and this is sufficient (\emph{via} a
    density argument) to use $\mathbf{1}_{\T^d}$ as test function
    which allows to recover \eqref{ineq:mass}.
  \item \textsf{\emph{entropy estimate:}} the $h_i$'s and $\mu_i$'s
    are locally Lipschitz, so boundedness of the $u_i$'s and their
    belonging to $\fsC^0([0,T];\fsL^2(\T^d))$ imply
    $h_i(u_i),\mu_i(u_i)\in\fsC^0([0,T];\fsL^1(\T^d))$, for
    $i=1,2$. With the same type of arguments we recover
    $h_i'(u_i) \in \fsL^2(0,T;\H^1(\T^d))$. This is sufficient to
    justify the following formula for all $t\in[0,T]$, by density of
    smooth functions,
    \begin{align*}
      \int_0^t \int_{\T^d} h_i'(u_i)\partial_t u_i = \int_{\T^d} h_i(u_i(t)) - \int_{\T^d} h_i(u_i^\init).
    \end{align*}
    Similarly, we have that (with the analogous formula for the other species)
    \begin{align*}
      -\int_0^t \int_{\T^d} h_1'(u_1) \lap\big(
      \big[\mu_1(u_2)\star\rho\big] u_1\big)
      = \int_0^t \int_{\T^d} h_1''(u_1) \nabla u_1 \cdot \nabla \big(\big[\mu_1(u_2)\star\rho\big] u_1\big),
    \end{align*}
    which is sufficient to reproduce rigorously the computation done
    in the proof of \cref{prop:formal} and integrate it time to get
    \eqref{ineq:ent}.
  \item \textsf{\emph{maximum principle:}} we have (using
    assumption~\eqref{ineq:assmu} and the entropy estimate)
    \begin{align*}
      \int_0^T
      \|\Delta\left(\mu_1(u_2)\conv\rho\right)(s)\|_{\fsL^\infty(\T^d)}
      \,\dd s
      &\leq  \|\mu_1(u_2)\|_{\fsL^1(Q_T)} \|\Delta\rho\|_{\fsL^\infty(\T^d)}\\
      &\leq \textnormal{A}T(1+ H^\init)\|\Delta\rho\|_{\fsL^\infty(\T^d)}
    \end{align*}
    and likewise for \(\lap(\mu_2(u_1)\conv \check{\rho})\) so that
    \eqref{ineq:max} follows directly from \cref{cor:kol1}.
  \item \textsf{\emph{duality estimate:}} the function $z:=u_1+u_2$ solves the following Kolmogorov equation
    \begin{align*}
      \partial_t z-\Delta(\mu z) &= 0,\\
      z(0,\cdot) &= u_1^\init+u_2^\init,
    \end{align*}
    with
    \begin{equation*}
      \mu:=\frac{(\mu_1(u_2)\star\rho) u_1
        +(\mu_2(u_1)\star\check{\rho}) u_2}{u_1+u_2},
    \end{equation*}
    where $\mu$ is well-defined thanks to the positivity of the $u_i$'s
    and furthermore bounded. The duality estimate of
    \cref{thm:kol} implies
    \begin{multline}\label{ineq:dualmost}
      \int_{Q_T} \big(\big[\mu_1(u_2)\star\rho\big] u_1
      + \big[\mu_2(u_1)\star\check{\rho}\big] u_2\big)(u_1+u_2) \\
      \lesssim_d \left(1+\int_{Q_T} \mu\right)\left(\int_{\T^d} (u_1^\init)^2 + \int_{\T^d} (u_2^\init)^2\right).
    \end{multline}
    To recover \eqref{ineq:dua}, simply notice that
    $\mu\leq \mu_1(u_2)\star\rho + \mu_2(u_1)\star\check{\rho}$ so that
    using the normalization of $\rho$ and assumption \eqref{ineq:assmu},
    \begin{align*}
      \int_{Q_T} \mu \leq \int_{Q_T} \mu_1(u_2) + \int_{Q_T} \mu_2(u_1)
      \leq \textnormal{A} \left(\int_{Q_T}(2+h_1(u_1)+h_2(u_2))\right)
      \leq 2\textnormal{A}B_{T,\init},
    \end{align*}
    where we used the entropy estimate and the constant
    $\textnormal{B}_{T,\init}:=T(1+H^\init)$.
  \end{itemize}

  These estimates have been proven for a positive solution with
  regularity \eqref{reg:ui} whose existence has been assumed. We now
  construct a solution by a fixed-point argument.

  On the set \(E := \fsL^1(Q_T) \times \fsL^1(Q_T)\) we define the map
  \(\Theta : E \to E\) which sends \((u_1,u_2)\) to the solutions
  \((u_1^\bullet,u_2^\bullet)\) (in the sense of \cref{defi:kol}) of
  \begin{equation*}
    \left\{
      \begin{lgathered}
        \partial_t u_1^{\bullet} = \lap\big(
        \big[\mu_1(u_2^+\wedge M)\star\rho\big] u_1^{\bullet}\big), \\
        \partial_t u_2^{\bullet} = \lap\big(
        \big[\mu_2(u_1^+\wedge M)\star\check{\rho}\big] u_2^{\bullet}\big),
      \end{lgathered}
    \right.
  \end{equation*}
  where the cutoff constant $M>0$ will be fixed later on. By
  continuity of $\mu_1$ and $\mu_2$ we have
  \begin{equation}
    \label{eq:torus:bound-mu}
    \max(\mu_1(x),\mu_2(x)) \le C,\qquad
    \forall x \in [0,M].
  \end{equation}
  In particular, \cref{thm:kol} applies and ensures that the previous
  map is well-defined. Moreover, \eqref{eq:torus:bound-mu} implies
  for $(u_1,u_2)\in E$ that
  \begin{align*}
    |\Delta\big(\mu_1(u_2^+\wedge M)\star \rho\big)| = |\mu_1(u_2^+ \wedge M)\star \Delta \rho| \leq C \|\Delta \rho\|_{\fsL^1(\T^d)},
  \end{align*}
  and likewise for \(\Delta\big(\mu_2(u_1^+\wedge M)\star
  \rho\big)\). Hence we infer from Corollaries \ref{cor:kol1} and
  \ref{cor:kol2} that the images \(u_1^\bullet\) and \(u_2^\bullet\)
  are non-negative and uniformly bounded in \(\fsL^\infty(Q_T)\) and
  \(\fsL^2(0,T;\H^1(\T^d))\). Moreover by the equations, the time
  derivatives are also uniformly bounded in
  \(\fsL^2(0,T;\H^{-1}(\T^d))\). That means that there exists a
  constant \(c\) such that
  \begin{align*}
    \Theta(E) \subset K
    : = \Big\{ &(v_1,v_2) \in E : \\
               &\max( \| v_i \|_{\fsL^\infty(Q_T)},
                 \| v_i \|_{\fsL^2(0,T;\H^1(\T^d))},
                 \| \partial_t v_i \|_{\fsL^2(0,T;\H^{-1}(\T^d))})
                 \le c
                 \text{ for } i=1,2
                 \Big\}.
  \end{align*}
  Then by the Aubin-Lions lemma the convex set $K$ is also compact in
  $E$. Hence Schauder's fixed-point theorem applies and ensures that
  there exists a fixed-point \((u_1,u_2)\in K\), solving therefore
  (the $u_i$'s are non-negative)
  \begin{equation}
    \label{eq:torus:existence-mapping}
    \left\{
      \begin{lgathered}
        \partial_t u_1 = \lap\big(
        \big[\mu_1(u_2\wedge M)\star\rho\big] u_1\big), \\
        \partial_t u_2 = \lap\big(
        \big[\mu_2(u_1\wedge M)\star\check{\rho}\big] u_2\big).
      \end{lgathered}
    \right.
  \end{equation}
  The bounds obtained for elements in \(K\) also ensure that \(u_1\) and \(u_2\) both belong to
  \(\fsC^0([0,T];\fsL^2(\T^d))\) so that we have the required
  regularity \eqref{reg:ui}. In order to conclude we just need to fix
  a constant $M$ such that the corresponding saturation vanishes. For
  this purpose, we consider
  \begin{equation*}
    M = 2 \gamma^{-1} \exp(\textnormal{AB}_{T,\init} \| \lap \rho \|_\infty),
  \end{equation*}
  where the constants $\textnormal{A}$ and $\textnormal{B}_{T,\init}$
  are defined in the statement of \cref{thm:existince-general-SKT} and
  we recall that $\gamma>0$ is such that
  \begin{align*}
    \gamma\leq u_i^\init\leq \gamma^{-1}.
  \end{align*}
  We now define
  \begin{equation*}
    t^\star := \sup \Big\{ t \in [0,T] :
    \max_{i=1,2}\| u_i \|_{\fsL^\infty([0,t]\times \T^d)} \le \frac{3M}{4} \Big\}.
  \end{equation*}
  By \cref{cor:kol1}, we have \(t^\star >0\) and up to any
  \(t \in (0,t^\star)\) the cutoff \(M\) has been irrelevant. Thus, for $t\in(0,t^\star)$ all the a priori estimates apply and in particular the entropy estimate which implies
  \begin{equation*}
    \| \mu_1(u_2) \|_{\fsL^1([0,t]\times \T^d)}
    \le \textnormal{A}\int_{0}^t \big(1+\int_{\T^d} h_2(u_2)\big)
    \leq \textnormal{A}t (1+H^\init)
    \leq \textnormal{AB}_{T,\init},
  \end{equation*}
  with a similar estimate for the other species.  This in turn implies
  by \cref{cor:kol1} that for $t<t^\star$
  \begin{equation*}
    \max_{i=1,2}  \| u_i \|_{\fsL^\infty([0,t]\times \T^d)} \le \frac{M}{2},
  \end{equation*}
  which proves that $t^\star=T$ by the usual continuity argument and
  our fixed-point $(u_1,u_2)$ is the required solution.
\end{proof}

\subsection{From non-local to local SKT}
We start with a compactness tool already used in
\cite{lepoutre-moussa-2017-entropic-structure-duality-multiple-species}
that we adapt slightly to our setting. The proofs are only included
for the reader's convenience.
\begin{lemma}\label{lem:com}
  Fix $\alpha:\pR \rightarrow \nnR$ having a negligible set of
  zeros. Consider a sequence of positive functions
  $(w_n)_n\in \textnormal{W}^{1,1}(Q_T)$ such that
  \begin{enumerate}[label={(\roman*)}]
  \item $(w_n)_n$  bounded in $\fsL^2(Q_T)$;
  \item $(\partial_t w_n)_n$ bounded in $\fsL^1(0,T;\H^{-m}(\mathbb{T}^d))$ for some integer $m$;
  \item $(\alpha(w_n)\nabla w_n)_n$ bounded in $\fsL^2(Q_T)$.
  \end{enumerate}
  Then $(w_n)_n$ admits an a.e. converging subsequence.
\end{lemma}
\begin{proof}
  By assumption $(iii)$ the sequence $(\nabla F(w_n))_n$ is bounded in
  $\fsL^2(Q_T)$, where $F:\nnR\rightarrow\nnR$ is defined by
  \begin{align*}
    F(z) := \int_0^z 1 \wedge \alpha.
  \end{align*}
  Moreover, $F$ is an increasing (because $\alpha>0$ a.e.)
  $1$-Lipschitz function vanishing at $0$. In particular, we infer
  from $(i)$ the same bound for $(F(w_n))_n$. Up to a subsequence we
  can thus assume that $(w_n)_n$ and $(F(w_n))_n$ respectively
  converge weakly to $w$ and $\widetilde{w}$ in $\fsL^2(Q_T)$.  Using
  $(ii)$ we thus infer from \cite[Proposition 3]{mou} that (up to a
  subsequence),
  \begin{align}
    \label{conv:comp}\int_{Q_T}  w_n F(w_n) \operatorname*{\longrightarrow}_{n\to\infty} \int_{Q_T}  w \widetilde{w}.
  \end{align}
  At this stage we use the Minty-Browder or Leray-Lions trick: one
  first establishes that
  \begin{multline*}
    \int_{Q_T} \stackrel{:=h_n}{\overbrace{(F(w_n)-F(w))(w_n-w)}} \\
    = \int_{Q_T} F(w_n)w_n+\int_{Q_T} F(w)w - \int_{Q_T}F(w_n)w-\int_{Q_T}F(w)w_n
    \operatorname*{\longrightarrow}_{n\to\infty} 0
  \end{multline*}
  by exploiting the $\fsL^2(Q_T)$ weak convergences
  $(w_n)_n \rightharpoonup_n w$,
  $(F(w_n))_n \rightharpoonup_n \widetilde{w}$, together with
  \eqref{conv:comp}. Then, since $F$ is increasing, we have
  $h_n\geq 0$ so that the previous convergence may be seen as the
  convergence of $(h_n)_n$ to $0$ in $\fsL^1(Q_T)$. In particular, up to
  some subsequence, we get that $(h_n)_n$ converges a.e.\ to 0 which in
  turn implies (increasingness of $F$) that $(w_n)_n\rightarrow w$.
\end{proof}
\begin{proof}[Proof of \cref{thm:sktasym}]
  Using the duality estimate of \cref{thm:existince-general-SKT} we
  first have
  \begin{align}\label{ineq:duan}
    \int_{Q_T} \Big(\big[\mu_1(u_{2,n})\star\rho_n\big] u_{1,n} +
    \big[\mu_2(u_{1,n})\star\check{\rho}_n\big] u_{2,n}\Big)\,
    (u_{1,n}+u_{2,n})
    \lesssim_{d,\init} 1,
  \end{align}
  where the constant depends on the dimension and initial data but is
  uniform in $n$. In particular, both species satisfy (since $\mu_1$
  and $\mu_2$ are positively lower-bounded) assumptions $(i)$ and
  $(ii)$ of Lemma~\ref{lem:com}. Using the entropy estimate of
  \cref{thm:existince-general-SKT}, we have also for both species that
  $(\alpha_i(u_{i,n})\nabla u_n)_n$ bounded in $\fsL^2(Q_T)$, which
  validates assumption $(iii)$ of the lemma since the dissipation
  rates $\alpha_i$ are assumed a.e. positive on $\pR$. We infer
  therefore from the previous lemma that, up to a subsequence (that we
  do not label), $(u_{1,n})_n$ and $(u_{2,n})_n$ converge a.e.\ to some
  $u_1$ and $u_2$, respectively.

  We now pass to the limit (in $\mathscr{D}'(Q_T)$) in the products
  \begin{align*}
    \big[\mu_1(u_{2,n})\conv\rho_n\big] u_{1,n} \text{ and } \big[\mu_2(u_{1,n})\conv\check{\rho}_n\big] u_{2,n}.
  \end{align*}
  W.l.o.g.\ we can focus on the first one.  Since $(u_{2,n})_n$
  converges to $u_2$ a.e., so does $(\mu_1(u_{2,n})_n$ to
  $\mu_1(u_2)$, by continuity of $\mu_1$.

  The assumption~\eqref{ineq:assmu} and the entropy estimate of
  \cref{thm:existince-general-SKT} imply that $(\mu_1(u_{2,n}))_n$ is
  bounded in $\fsL^\infty(0,T;\fsL^1(\T^d))$. As this is not
  sufficient to prevent possible concentration in the space variable,
  we use the growth assumption \eqref{lim:mubound}, to establish the
  uniform integrability of $(\mu_1(u_{2,n}))_n$. Indeed, since $\mu_1$
  is continuous, the sequence
  $c_R:= \inf\{z\geq 0\,:\,\mu_1(z) \geq R\}$ diverges to $+\infty$
  with $R$ and we have
  \begin{align*}
    \int_{Q_T} \mu_1(u_{2,n}) \mathbf{1}_{\mu_1(u_{2,n})\geq R} & \leq     \int_{Q_T} \mu_1(u_{2,n}) \mathbf{1}_{u_{2,n}\geq c_R} \leq \sup_{z \geq c_R} \Phi(z) \int_{Q_T} h_2(u_{2,n}) + u_{2,n}^2,
  \end{align*}
  where $\Phi(z):= \frac{\mu_1(z)}{h_2(z)+ z^2}$ goes to $0$ as
  $z\rightarrow +\infty$, by assumption \eqref{lim:mubound}. Since
  $\mu_2$ is positively lower-bounded (thanks to
  assumption~\eqref{ineq:assmu}), we infer from \eqref{ineq:duan} a
  bound for $(u_{2,n})_n$ in $\fsL^2(Q_T)$ (we use here the
  non-negativity of all the involved functions). Using the entropy
  estimate and the previous inequalities, we therefore infer
  \begin{align*}
    \lim_{R\rightarrow +\infty} \sup_{n}     \int_{Q_T} \mu_1(u_{2,n}) \mathbf{1}_{\mu_1(u_{2,n})\geq R} =0,
  \end{align*}
  which establishes uniform integrability.

  Therefore, Vitali's convergence theorem implies that
  $(\mu_1(u_{2,n}))_n$ converges to $\mu_1(u_2)$ in $\fsL^1(Q_T)$. The
  sequence $(\mu_1(u_{2,n})\star\rho_n)_n$ shares the same
  behaviour. In particular, $(\mu_1(u_{2,n})\star\rho_n)_n$ is also
  uniformly integrable and adding a subsequence if necessary, we can
  assume that it converges a.e. towards $\mu_1(u_2)$. Now, to conclude
  we write
  \[w_n:=\big[\mu_1(u_{2,n})\conv\rho_n\big] u_{1,n} =
    \big[\mu_1(u_{2,n})\conv\rho_n\big]^{1/2}\,\big[\mu_1(u_{2,n})\conv\rho_n\big]^{1/2}
    u_{1,n}.
  \]
  As already noticed, $(w_n)_n$ converges a.e. to the expected limit
  $\mu_1(u_2)u_1$. The previous writing together with the duality
  estimate \eqref{ineq:duan} and the Cauchy-Schwarz inequality shows
  that $(w_n)_n$ is bounded in $\fsL^1(Q_T)$. Even better, $(w_n)_n$
  is the product of a $\fsL^2$-uniformly integrable sequence with an
  $\fsL^2$-bounded one so that $(w_n)_n$ is uniformly integrable
  and the Vitali convergence theorem applies once more to get the
  convergence of $(w_n)_n$ towards $\mu_1(u_2)u_1$.

  The previous reasoning (which applies to both species) allows to
  pass to the limit of the equations. The limit satisfies the
  estimates by Fatou's lemma.
\end{proof}

\section{General regularised scheme on a domain}
\label{sec:solution-regularised-system}

In this section, we study the general regularisation scheme introduced
in \cref{thm:formal-general} and prove the corresponding results
\cref{thm:general-system-existence-regularised} and
\cref{thm:general-skt-asymptotic}.

\subsection{Existence of regularised solutions}

We start with proving the existence of solutions for the regularised
scheme, i.e.\ \cref{thm:general-system-existence-regularised}.

The advantage of the regularisation is that the cross-diffusion terms
are controllable and we thus rewrite the evolution as
\begin{equation*}
  \begin{split}
    &\partial_t u_i(x_i)
    - \divergence_{x_i}
    \left[
      \left(
        \epsilon +
        \prod_{k\not = i} \int_{x_k \in \Omega} \dd x_k\,
        K(x_1,\dots,x_n)\,
        a_{ii}(u_1(x_1),\dots,u_n(x_n))
      \right)
      \nabla u_i(x_i)
    \right] \\
    &=
    \divergence_{x_i}
    \left[
      \prod_{k\not = i} \int_{x_k \in \Omega} \dd x_k\,
      K(x_1,\dots,x_n)\,
      \sum_{j\not= i}
      a_{ij}(u_1(x_1),\dots,u_n(x_n))
      \nabla u_j(x_j)
    \right].
  \end{split}
\end{equation*}

For the cross-diffusion terms, the \(\tilde{a}_{ij}\) in
\cref{thm:general-system-existence-regularised} are defined such that
\begin{equation*}
  a_{ij}(u_1(x_1),\dots,u_n(x_n))
  \nabla u_j(x_j)
  = \nabla_{x_j} \tilde{a}_{ij}(u_1(x_1),\dots,u_n(x_n))
\end{equation*}
so that the partial derivative can formally be integrated by parts onto the
kernel \(K\), where no boundary terms appear due to
\eqref{eq:general-regularised-evolution-boundary}. Hence the evolution
can be rewritten as
\begin{equation}
  \label{eq:general:parabolic-form}
  \partial_t u_i - \nabla\big( (\epsilon + \bar{a}_i[u]) \nabla
  u_i\big)
  + \bar{b}_i[u] \nabla u_i + \bar{c}_i[u] u_i = 0,
\end{equation}
with von Neumann boundary conditions and
\begin{align}
  \bar{a}_i[u](x_i)
  &=
    \prod_{k\not = i} \int_{x_k \in \Omega} \dd x_k\,
    K(x_1,\dots,x_n)\,
    a_{ii}(u_1(x_1),\dots,u_n(x_n)),\\
  \bar{b}_i[u](x_i)
  &= \sum_{j\not =i}
    \prod_{k\not = i} \int_{x_k \in \Omega} \dd x_k\,
    \partial_j K(x_1,\dots,x_n)\,
    \partial_i \tilde{a}_{ij}(u_1(x_1),\dots,u_n(x_n)),\\
  \bar{c}_i[u](x_i)
  &= \sum_{j\not =i}
    \prod_{k\not = i} \int_{x_k \in \Omega} \dd x_k\,
    \partial_{ij} K(x_1,\dots,x_n)\,
    \frac{\tilde{a}_{ij}(u_1(x_1),\dots,u_n(x_n))}{u_i(x_i)}.
\end{align}
The assumptions of \cref{thm:general-system-existence-regularised}
then imply for \(x_i \in \domain\) that
\begin{equation}
  \label{eq:general-coefficient-bounds}
  \begin{split}
    |\bar{a}_i[u](x_i)|
    &\le A\, \| K \|_{\infty} (|\domain| + H(u)),\\
    |\bar{b}_i[u](x_i)|
    &\le A\, \| \nabla K \|_{\infty} (|\domain| + H(u)),\\
    |\bar{c}_i[u](x_i)|
    &\le A\, \| \nabla^2 K \|_{\infty} (|\domain| + H(u)).
  \end{split}
\end{equation}

This is enough to prove the existence of solutions by a Galerkin
scheme.

\begin{proof}[Proof of \cref{thm:general-system-existence-regularised}]
  Let \(\sigma \in \fsC^\infty_c(\R^n)\) a non-negative mollification
  kernel with \(\supp \sigma \subset B_1\) and
  \(\int \sigma \dd x = 1\) and define
  \begin{equation*}
    \sigma^m(x) = m^d\, \sigma(mx).
  \end{equation*}
  Extending \(\bar{a}_i\), \(\bar{b}_i\), \(\bar{c}_i\) with zero
  outside \(\domain\), we consider for \(m \in \N\) the following
  system
  \begin{equation}
    \label{eq:mollified-galerkin}
    \begin{split}
      \partial_t u_i^m &- \nabla\big( (\epsilon
      + ((\bar{a}_i[u^m]\wedge M) \conv \sigma^m) \nabla
      u_i^m\big) \\
      &+ ((\bar{b}_i[u^m]\wedge M) \conv \sigma^m) \nabla u_i^m
      + ((\bar{c}_i[u^m]\wedge M) \conv \sigma^m)  u_i^m = 0,
    \end{split}
  \end{equation}
  with von Neumann boundary conditions, \(i=1,\dots,n\) and the
  constant \(M\) as in \eqref{eq:growth-constant-m}.

  By a standard Galerkin scheme (e.g.\ taking the von Neumann
  eigenvectors of the Laplacian on \(\domain\)), the system
  \eqref{eq:mollified-galerkin} has a solution \(u_i^m\) with initial
  data \(u_i^\init\) and has any \(H^k\), \(k\in\N\), regularity after
  an arbitrary short time. Hence we can apply the maximum principle
  for parabolic equations and find as in \cref{cor:kol1} that
  \begin{equation*}
    \gamma \exp\left(-TM\right)
    \leq u_i^m
    \leq
    \gamma^{-1} \exp\left(TM\right).
  \end{equation*}
  Furthermore, each \(u_i^m\) is preserving the mass. Finally, we can
  test \eqref{eq:mollified-galerkin} against \(u_i^m\) to find the
  followig estimate independent of \(m\):
  \begin{equation*}
    \sup_{t\in [0,T]} \| u_i^m(t,\cdot) \|_{\fsL^2(\domain)}^2
    + \epsilon \int_0^T \| \nabla u_i^m(t,\cdot) \|_{\fsL^2(\domain)}^2
    \le
    \exp\left[TM
      \Big(
      2+\frac{1}{\epsilon}
      \Big)
    \right]
    \| u_i^\init \|_{\fsL^2(\domain)}^2,
  \end{equation*}
  where \(i=1,\dots,n\) and we used the cutoff with \(M\). Note that
  here the RHS is bounded by assumption. Hence we find for a constant
  \(C(T)\) independent of \(m\) that
  \begin{equation*}
    \| \partial_t u_i^m \|_{\fsL^2(0,T,\H^{-1}(\domain))} \le C(T)
  \end{equation*}
  for \(i=1,\dots,n\).

  By Aubin-Lions lemma we can therefore find a subsequence
  (relabelling with \(m\)) and
  \begin{equation*}
    u_i \in
    \fsC^0([0,T];\fsL^2(\domain))\cap \fsL^2(0,T;\H^1(\domain))\cap\fsL^\infty(0,T;\fsL^\infty(\domain)),
  \end{equation*}
  for \(i=1,\dots,n\) such that \(u_i^m\) converges almost everywhere
  to \(u_i\) and \(\nabla u_i^m\) converges \(\fsL^2\) weakly to
  \(\nabla u_i\). Moreover, it holds that
  \begin{equation*}
    \gamma \exp\left(-TM\right)
    \leq u_i
    \leq
    \gamma^{-1} \exp\left(TM\right).
  \end{equation*}
  The convergence implies that for \(\phi \in \fsC^\infty(\domain_T)\)
  with \(\phi(T,\cdot)\equiv 0\) it holds that
  \begin{equation*}
    \begin{split}
      -\int_0^T \int_\domain
      u_i \partial_t \phi
      &+\int_0^T \int_\domain
      (\epsilon + \bar{a}_i[u]\wedge M) \nabla u_i
      \cdot \nabla \phi
      +\int_0^T \int_\domain
      (\bar{b}_i[u]\wedge M) \cdot \nabla u_i\;
      \phi\\
      &+\int_0^T \int_\domain
      (\bar{c}_i[u]\wedge M) u_i\;
      \phi
      =\int_\domain u_i^\init \phi^\init,
    \end{split}
  \end{equation*}
  i.e.\ \(u=(u_1,\dots,u_n)\) is a weak solution with von Neumann
  boundary data. Moreover, by the continuity this implies directly the
  conservation of mass.

  Until a time \(T^* \le T\) for which
  \begin{equation*}
    \sup_{t\in [0,T^*]}
    \sup_{x\in \domain}
    \max(\bar{a}_i[u],|\bar{b}_i[u]|,\bar{c}_i[u])
    \le M,
  \end{equation*}
  the cutoff \(M\) is not applied and we have a weak solution of
  \eqref{eq:general:parabolic-form}.  As in the Laplace case on the
  torus in \cref{sec:convo}, the proven regularity is sufficient to
  justify rigorously the formal entropy estimate as in
  \cref{thm:formal-general}.

  The assumptions \eqref{eq:general-coefficient-bounds} then imply
  that at time \(T^*\) it holds that
  \begin{equation*}
    \sup_{x\in \domain}
    \max(\bar{a}_i[u],|\bar{b}_i[u]|,\bar{c}_i[u])
    \le \frac{M}{2}
  \end{equation*}
  and thus by continuity \(T^*=T\) and we have constructed the claimed
  solution.
\end{proof}

\subsection{Limit for the SKT system}

Having constructed the nonlocal approximation, we now prove
\cref{thm:general-skt-asymptotic}.

The assumption of the extension operator allows to find a uniform
Gagliardo-Nirenberg inequality.
\begin{lemma}
  \label{thm:uniform-gagliardo-nirenberg}
  Assume the setup of \cref{thm:general-skt-asymptotic}. Then there
  exists a uniform \(c\) for the Gagliardo-Nirenberg inequality
  \begin{equation*}
    \| f \|_{\fsL^p(A_m)}^p
    \le c\, \Big(
    \| f \|_{\fsL^1(A_m)}^{(1-\theta)p}
    \| \nabla f \|_{\fsL^2(A_m)}^{\theta p}
    +
    \| f \|_{\fsL^1(A_m)}^{p}
    \Big)
    \qquad
    \forall f:A_m\to \R,
  \end{equation*}
  holds on all \(A_m\), where \(m\in \N\), \(\theta=2/p\)
  and \(p=2+2/d\).
\end{lemma}
\begin{proof}
  By the extension operator \(\sExt_m\) and the Gagliardo-Nirenberg
  inequality on \(\R^d\) we find 
  \begin{equation*}
    \begin{split}
      \| f \|_{\fsL^p(A_m)}^p
      &\le \| \sExt_m(f) \|_{\fsL^p(\R^d)}^p \\
      &\lesssim \| \sExt_m(f) \|_{\fsL^1(\R^d)}^{(1-\theta)p}\;
      \| \nabla \sExt_m(f) \|_{\fsL^2(\R^d)}^2 \\
      &\lesssim \| f \|_{\fsL^1(A_m)}^{(1-\theta)p}\;
      \big(\| f \|_{\fsL^2(A_m)}^2
      + \| \nabla f \|_{\fsL^2(A_m)}^2 \big).
    \end{split}
  \end{equation*}
  As \(p>2\), we can interpolate \(\| f \|_{\fsL^2(A_m)}\) between
  \(\| f \|_{\fsL^1(A_m)}\) and \(\| f \|_{\fsL^p(A_m)}\) and absorb
  the contribution of \(\| f \|_{\fsL^p(A_m)}\) so that the claimed
  inequality follows.
\end{proof}

The first lemma ensures the
integrability and determines the sequence \(\epsilon\).

\begin{lemma}
  \label{thm:choice-epsilon-skt}
  Assume the setup of \cref{thm:general-skt-asymptotic}. Then there
  exists a constant \(C_T\) and a decreasing sequence
  \((\epsilon_m)_{m}\) with \(\epsilon_m \downarrow 0\) such that
  \begin{equation*}
    \| u^m_i \|_{\fsL^{\tilde{p}}([0,T)\times\domain)} \le C_T
  \end{equation*}
  for \(i=1,\dots,n\) and \(\tilde{p} = 2+1/d\).
\end{lemma}
\begin{proof}
  For the regularisation kernel \(K^m\) and \(\epsilon_m > 0\), we find
  by \cref{thm:general-system-existence-regularised} a solution
  \(u_m\) which satisfies
  \begin{enumerate}[label={(\roman*)}]
  \item \(\|u_i^m\|_{\fsL^\infty(0,T;\fsL^1(\domain))} \le c\) for
    \(i=1,\dots,n\) (conservation of mass),
  \item \(\|\nabla u_i^m\|_{\fsL^2(0,T;\fsL^2(\domain))}^2 \le
    \epsilon_m^{-1} \exp(\epsilon_m^{-1}) c\) (\(\epsilon\)-dependent
    estimate),
  \item \(\|\nabla u_i^m\|_{\fsL^2(0,T;\fsL^2(A_m))}^2 \le c\)
    (dissipation estimate in the set \(A_m\) on which the weights are
    \(w_i^m \equiv 1\))
  \end{enumerate}
  for a constant \(c\) independent of \(m\).

  The parameters \(\theta\) and \(p\) of the Gagliardo-Nirenberg in
  \cref{thm:uniform-gagliardo-nirenberg} are chosen such that
  \(\theta p = 2\). Hence we find on \(A_m\) that for \(i=1,\dots,n\)
  \begin{equation*}
    \int_0^T \| u_i^m \|_{\fsL^p(A_m)}^p \dd t
    \lesssim \int_0^T
    \left(\| \nabla u_i^m \|_{\fsL^2(A_m)}^2\;
      \| u_i^m \|_{\fsL^1(A_m)}^{(1-\theta)p}
      +
      \| u_i^m \|_{\fsL^1(A_m)}^{p}
    \right)\; \dd t.
  \end{equation*}
  With the gradient control from the dissipation and the conservation
  of mass this shows
  \begin{equation*}
    \int_{[0,T)\times A_m} |u_i^m|^{p}\, \dd x\, \dd t
    \le c_d
  \end{equation*}
  for a constant \(c_d\) independent of \(m\).

  As the domain \(\domain\) is assumed to have \(\fsC^1\) boundary, we
  can also apply the argument of
  \cref{thm:uniform-gagliardo-nirenberg} to find over \(\domain\) that
  for \(i=1,\dots,n\)
  \begin{equation*}
    \int_0^T \| u_i^m \|_{\fsL^p(\domain)}^p \dd t
    \lesssim \int_0^T
    \left(\| \nabla u_i^m \|_{\fsL^2(\domain)}^2\;
      \| u_i^m \|_{\fsL^1(\domain)}^{(1-\theta)p}
      +
      \| u_i^m \|_{\fsL^1(\domain)}^{p}
    \right)\; \dd t.
  \end{equation*}
  Hence we find for a constant \(c_e\) independent of \(m\) that
  \begin{equation*}
    \int_{[0,T)\times\domain} |u_i^m|^{p}\, \dd x\, \dd t
    \le c_e \big(
    1+ \frac{c_e}{\epsilon_m} \exp(\epsilon_m^{-1})\big).
  \end{equation*}

  As \(\tilde{p} < p\) we can find \(q\in (0,1)\) so that the
  Hölder inequality implies
  \begin{equation*}
    \| f \|_{\fsL^{\tilde{p}}([0,T)\times B)} \le
    (T\, |B|)^q\; \| f \|_{\fsL^p([0,T)\times B)}
  \end{equation*}
  for \(B \subset \domain\) and \(f \in \fsL^p([0,T)\times B)\).

  By splitting \(\domain\) into \(A_m\) and \(\domain\setminus A_m\)
  we therefore find (as \(|A_m| \le |\domain|\))
  \begin{equation*}
    \begin{split}
      \| u_i^m \|_{\fsL^{\tilde{p}}([0,T) \times \domain)}
      &\le \| u_i^m \|_{\fsL^{\tilde{p}}([0,T) \times A_m)}
      + \| u_i^m \|_{\fsL^{\tilde{p}}([0,T) \times (\domain\setminus A_m))} \\
      &\le T^q\, |\domain|^q c_d^{1/p}
      + T^q\, |\domain\setminus A_m|^q c_e^{1/p}
      \left( 1 + \frac{\exp(\epsilon_m^{-1})}{\epsilon_m} \right)^{1/p}.
    \end{split}
  \end{equation*}
  As \(|\domain\setminus A_m| \to 0\) and \(q \in (0,1)\), we can
  therefore find a sequence \(\epsilon_m \downarrow 0\) such that
  \(|\domain\setminus A_m|^q (1 + \epsilon_m^{-1}
  \exp(\epsilon_m^{-1}))^{1/p}\) is bounded by a constant independent
  of \(m\). The claim then follows directly from the given estimate.
\end{proof}

We can now proceed with the convergence result.
\begin{proof}[Proof of \cref{thm:general-skt-asymptotic}]
  Inside each good set \(A_{\bar{m}}\), the dissipation and mass
  conservation give a uniform estimate for \(u_i^m\) in
  \(\fsL^\infty(0,T;\fsL^1(A_m))\) and \(\fsL^2(0,T;\H^1(A_m))\) for
  \(i=1,\dots,n\) and \(m\ge \bar{m}\). By the equation this also
  gives a uniform estimate of the time-derivative in
  \(\fsL^1(0,T;\H^{-k}(A_m))\) for a large enough \(k \in \N\)
  (depending only on dimension \(d\)). Hence on \(A_{\bar{m}}\) we
  have compactness for \(u_i^m\). As \(A_m \uparrow \domain\), a
  diagonal argument shows that along a subsequence (which we relabel
  with \(m\)) that for \(i=1,\dots,n\) there exist
  \(u_i : [0,T) \times \domain\) such that \(u^m_i \to u_i\)
  a.e. Moreover, choosing \(\epsilon_m\) as in
  \cref{thm:choice-epsilon-skt} we find
  \(u_i \in \fsL^{\tilde{p}}([0,T)\times\domain)\).

  By the dissipation inequality we find that
  \begin{equation*}
    \| \sqrt{w_i^m} \nabla u_i^m \|_{\fsL^2([0,T)\times\domain)}
  \end{equation*}
  is uniformly bounded. Hence along a subsequence
  \(\sqrt{w_i^m} \nabla u_i^m\) converges weakly in \(\fsL^2\) to a limit
  \(\psi_i\). As \(w_i^m\) is the constant \(1\) inside the set \(A_m\)
  and \(A_m \uparrow \Omega\), it follows that
  \(\sqrt{w_i^m} \nabla u_i^m \weakto \nabla u\) and \(\nabla u \in \fsL^2\).

  As \(u^m\) preserves the mass, is non-negative and satisfies the
  entropy-dissipation inequality, the same is true for the limit \(u\)
  by using the stated regularity. Moreover, the stated regularity
  gives the claimed convergence.

  It thus remains to check that \(u\) is a weak solution. As \(u^m\)
  satisfies von Neumann boundary data and \(K^m\) vanishes at the
  boundary, the constructed solutions satisfy for all
  \(\phi \in \fsC^\infty([0,T] \times \domain\) with
  \(\phi(T,\cdot) \equiv 0\) and \(i=1,\dots,n\) that
  \begin{equation*}
    \begin{split}
      &- \int_0^T \int_\domain u_i^m\, \partial_t \phi
      + \epsilon_m \int_0^T \int_\domain
      \nabla u_i^m \cdot \phi \\
      &+ \int_0^T \int_\domain
      \left(\prod_{k\not = i} \int_{x_k\in \domain} \dd x_k
        K^m(x_1,\dots,x_n) \sum_{j=1}^{n} a_{ij}(u_1^m(x_1),\dots,u_n^m(x_n)) \nabla u_j^m(x_j)\right) \cdot
      \nabla \phi \\
      &= \int_{\domain} u_i^\init \phi(0,\cdot).
    \end{split}
\end{equation*}

  For the diffusion from \(d_{ij}\) with \(i\not= j\) and
  \(i,j=1,\dots,n\) we must therefore show that for all test function
  \(\phi\)
  \begin{equation*}
    \begin{split}
      &\int_0^T \int_{\domain^n}
      K^m(x_1,\dots,x_n) u_i^m(t,x_i) \nabla u_j^m(t,x_j)\, \nabla \phi(t,x_i)\,
      \dd x_1\dots \dd x_n \dd t \\
      &\to
      \int_0^T
      \int_\domain u_i(t,x) \nabla u_j(t,x)\, \nabla \phi(t,x)\,
      \dd x\, \dd t.
    \end{split}
  \end{equation*}
  We rewrite the nonlinear diffusion term as
  \begin{equation*}
    \begin{split}
      &\int_0^T \int_{\domain^n}
      K^m(x_1,\dots,x_n) u_i^m(t,x_i) \nabla u_j^m(t,x_j)\, \nabla \phi(t,x_i)\,
      \dd x_1\dots \dd x_n \dd t \\
      &= \int_0^T \int_{\domain} \sqrt{w_j^m(x_j)} \nabla u_j^m(t,x_j) \psi^m(t,x_j)\,
      \dd x_i\, \dd t,
    \end{split}
  \end{equation*}
  where
  \begin{equation*}
    \psi^m(t,x_j) =
    \frac{1}{\sqrt{w_j^m(x_j)}}
    \prod_{k\not =j} \int_{\domain} \dd x_k\,
    K^m(x_1,\dots,x_n) u_i^m(t,x_i)\, \nabla \phi(t,x_i).
  \end{equation*}
  By the definition of the weight \(w_j^m\), we can apply Jensen's inequality
  for \(\tilde{p}\ge 2\) to find that
  \begin{equation*}
    \| \psi^m \|_{\fsL^{\tilde{p}}([0,T)\times \domain)}
    \le
    \| u_i^m \nabla \phi \|_{\fsL^{\tilde{p}}([0,T)\times \domain)}.
  \end{equation*}

  By the proven convergence and regularity of \(u_i^m\) we find that
  \(\psi^m \to u_i \nabla \phi\) a.e.\ in \([0,T) \times
  \domain\). The previous inequality gives a uniform bound of
  \(\psi^m\) in \(\fsL^{\tilde{p}}\) with \(\tilde{p}>2\) so that
  \(\psi^m\) converges strongly in \(\fsL^2\) to \(u_i \nabla \phi\). As
  \(\sqrt{w_j^m} \nabla u_j^m\) converges weakly in \(\fsL^2\) to
    \(\nabla u_j\), this proves the claimed convergence.

  The other terms in the weak formulation converge more directly in
  the limit and we thus have found a weak solution.
\end{proof}

\appendix

\section{Microscopic reversibility}
\label{sec:reversibility}

In the linear SKT model \eqref{eq:local-SKT}, the entropy was
understood as reversiblity in a microscopic model in
\cite{daus-desvillettes-dietert-2019-entropic-structure} and this gave
us the intuition about the nonlocal entropy structure. In this
appendix we discuss in the case of two species how the form in
\cref{rem:laplace-general-regularisation} in the general
regularisation on bounded domains by a kernel
\(K : \domain^2 \to \R_{\ge 0}\) appears formally from the microscopic
entropy structure.

In the microscopic picture of
\cite{daus-desvillettes-dietert-2019-entropic-structure} we considered
a spatial discretisation in the one-dimensional setting so that we
have discrete positions \(\{1,\dots,N\}\). On this discrete setting we
consider many particles of the two species \(1\) and \(2\) and we then
obtain a reversible cross-diffusion behaviour if a pair consisting of
a particle of species \(1\) at position \(i\) and a particle of
species \(2\) at position \(j\) jumps together with rate \(R_r(i,j)\)
to the positions \(i+1\) and \(j+1\), respectively. Likewise the pair
can jump with a rate \(R_l(i,j)\) to \(i-1\) and \(j-1\),
respectively. We then have the reversibility (and thus the entropy
structure) if
\begin{equation} \label{eq:microsopic-reversibility}
  R_r(i,j) = R_l(i+1,j+1).
\end{equation}

In the formal mean-field limit we then find the evolution for the
densities \(u_1\) and \(u_2\) the following nonlinear system
\begin{equation*}
  \left\{
    \begin{aligned}
      \partial_t u_1(i) &= \sum_{j=1}^M
      \Big\{
      R_r(i-1,j) u_1(i-1) u_2(j) + R_l(i+1,j) u_1(i+1) u_2(j) \\
      &\qquad\qquad
        - (R_l(i,j)+R_r(i,j)) u_1(i) u_2(j)
      \Big\} \\
      \partial_t u_2(j) &= \sum_{i=1}^M
      \Big\{
      R_r(i,j-1) u_1(i) u_2(j-1) + R_l(i,j+1) u_1(i) u_2(j+1) \\
      &\qquad\qquad
        - (R_l(i,j)+R_r(i,j)) u_1(i) u_2(j)
      \Big\}
    \end{aligned}
  \right.
\end{equation*}
for which we can indeed verify the entropy
\begin{equation*}
  H = \sum_{i=1}^M
  \Big[ h(u_1(i)) + h(u_2(i)) \Big]
\end{equation*}
where \(h'(x) = \log x\) as
\begin{equation*}
  \frac{\dd}{\dd t} H
  = - \sum_{i,j} R_r(i,j)
  \Big[
    \big(u_1(i+1)u_2(j+1) - u_1(i) u_2(j)\big)
    \big(\log(u_1(i+1) u_2(j+1)) - \log(u_1(i)u_2(j)) \big)
  \Big].
\end{equation*}

For the formal limit of the discrete system to a PDE, we denote the
centred discrete Laplacian
\begin{equation*}
  (\lap_d f)(i) = f(i+1) + f(i-1) - 2f(i).
\end{equation*}
We can then rewrite the evolution as
\begin{equation*}
  \begin{aligned}
    &\partial_t u_1(i)
    = \lap_d\left(\sum_j \frac{R_l(i,j) + R_r(i,j)}{2}
      u_1(i) u_2(j) \right) \\
    &+
    \frac 12
    \sum_j
    \left\{
      u_1(i{+}1) u_2(j)
      \Big[
      R_l(i{+}1,j) - R_r(i{+}1,j)
      \Big]
      + u_1(i{-}1) u_2(j)
      \Big[
      R_r(i{-}1,j) - R_r(i{+}1,j)
      \Big]
    \right\}
  \end{aligned}
\end{equation*}
and likewise for \(u_2\). By the microscopic reversibility
\eqref{eq:microsopic-reversibility} we note that this is exactly the
discrete form of the regularisation found in
\eqref{eq:regularised-skt-general-k}.

\section*{Acknowledgements}

We would like to thank Luca Alasio and Markus Schmidtchen for the
organisation of the workshop ``Recent Advances in Degenerate Parabolic
Systems with Applications to Mathematical Biology'' at Laboratoire
Jacques-Louis Lions (LJLL) in Paris in 2020, where the first
discussion on the project emerged.

\printbibliography

\end{document}